\documentclass[twoside,reqno]{amsart}
\usepackage[colorlinks=true,citecolor=blue]{hyperref}
\usepackage{mathptmx, amsmath, amssymb, amsfonts, amsthm, mathptmx, enumerate, color,amssymb, mathrsfs, upgreek }
\usepackage{graphicx}
\usepackage{epstopdf}
\newtheorem{theorem}{Theorem}[section]
\newtheorem{lemma}[theorem]{Lemma}
\newtheorem{proposition}[theorem]{Proposition}

\theoremstyle{definition}
\newtheorem{definition}[theorem]{Definition}

\numberwithin{equation}{section}

\begin{document}
	\setcounter{page}{1}	
	
	\vspace*{2.0cm}
	\title[Generalized noncooperative Schr\"{o}dinger-Kirchhoff systems in $\mathbb{R}^N$ ]
	{Generalized noncooperative Schr\"{o}dinger-Kirchhoff-type systems in $\mathbb{R}^N$}
	\author[N. Chems Eddine and D. D. Repov\v{s}]{Nabil CHEMS EDDINE $^{1}$ and
	 Du\v{s}an D. Repov\v{s} $^{2,3,4*}$}
	\maketitle
	\vspace*{-0.6cm}	
	
	\begin{center}
		{\footnotesize
		$^1$Laboratory of Mathematical Analysis and Applications, Department of Mathematics, Faculty of Sciences, Mohammed V University, P.O. Box 1014, Rabat, Morocco.\\
		$^2$Faculty of Education, University of Ljubljana, 1000 Ljubljana, Slovenia.\\
			$^3$Faculty of Mathematics and Physics, University of Ljubljana, 1000 Ljubljana, Slovenia.\\
			$^4$ Institute of Mathematics, Physics and Mechanics, 1000 Ljubljana, Slovenia.
			
	}\end{center}	
	
	\vskip 4mm {\footnotesize \noindent {\bf Abstract.}		
		
	  We consider a class of noncooperative Schr\"{o}dinger-Kirchhoff type system which involves a general variable exponent elliptic operator with critical growth. Under certain suitable conditions on the nonlinearities, we establish the existence of infinitely many solutions for the problem by using the limit index theory, 
	  a version of concentration-compactness principle for weighted-variable exponents Sobolev spaces
	   and  the principle of symmetric criticality of Krawcewicz and Marzantowicz.\\		
		
		
		\noindent {\it Math. Subj. Classif.}
		35B20, 35B33, 35D30, 35J50, 35J60, 46E35.}	
	
\renewcommand{\thefootnote}{}
\footnotetext{$^*$Corresponding author: Du\v{s}an D. Repov\v{s}
	\par
	E-mail addresses: nab.chemseddine@gmail.com (N. Chems Eddine), 
	dusan.repovs@guest.arnes.si (D. D. Repov\v{s})}
	
	\section{Introduction}
	
		 The purpose of this paper is to investigate the multiplicity of solutions for  the noncooperative Schr\"{o}dinger-Kirchhoff-type systems involving a general variable exponent elliptic operator and critical nonlinearity in $\mathbb{R}^N$: 
	
	\begin{eqnarray}
		\label{s1.1}
		\begin{cases}
			K\Big(\mathscr{B}(u)\Big)\Big(\textrm{div}\,\Big(  \mathscr{A}_1(\nabla u)\Big)-b(x)\mathcal{A}_2(u) \Big)=|u|^{r (x)-2}u+\uplambda(x) \frac{ \partial \mathcal{F}}{\partial u}(x,u,v) & \text{in } \mathbb{R}^N, \\
			K\Big(\mathcal{B}(v)\Big)\Big(-\textrm{div}\,\Big(  \mathcal{A}_1(\nabla v)\Big)+b(x)\mathcal{A}_2(v)\Big) =|v|^{r (x)-2}v+\uplambda(x
			) \frac{ \partial \mathcal{F}}{\partial v}(x,u,v) & \text{in } \mathbb{R}^N,\\
		~~~~~~~~u,v \in W^{1,p(x)}_{b}(\mathbb{R}^N)\cap W^{1,h(x)}_{b}(\mathbb{R}^N),
		\end{cases}
	\end{eqnarray}	
	where $N\geq 2$, $\uplambda$ is a
	continuous, radially symmetric function on $\mathbb{R}^N$,  $b\in L^{\infty}(\mathbb{R}^N)$ satisfies
	 $b_0 := \rm{ess} \inf_{} \{b(x) : x\in \mathbb{R}^N\} > 0$, $\nabla \mathcal{F}=( \frac{ \partial \mathcal{F}}{\partial u},\frac{ \partial \mathcal{F}}{\partial v})$ is the gradient of a $C^1$ function $\mathcal{F}: \mathbb{R}^N\times\mathbb{R}^2\to \mathbb{R}^+$, the functions $p$ and $q$ are log-Holder continuous, radially symmetric on $\mathbb{R}^N$, and satisfy the following inequality
	\begin{gather}
		1 <p^{-}\leq p(x) \leq p^{+}< q^{-}\leq q(x) \leq q^{+}
		< N, 
	\end{gather}
	for all $x\in \mathbb{R}^N,$ and the function $s$ is continuous, radially symmetric on $\mathbb{R}^N$ and satisfies the following inequality
	\begin{gather}
		h^{-}\leq h(x) \leq r^{-}\leq r(x) \leq r^{+}\leq h^{*}(x)<\infty,
	\end{gather}	
	for all $x\in \mathbb{R}^N,$ where $p^-:= \text{ess inf}_{}\{p(x) : x\in \mathbb{R}^N\}$,  $p^{+}:= \text{ess sup}_{}\{p(x) : x\in \mathbb{R}^N\}$, and analogously for
	$q^-, q^+, h^-$,
	$h^+, r^-$ and $r^+$, with $h(x) =(1-\mathcal{H}(\kappa^3_{\star})) p(x)+ \mathcal{H}(\kappa^3_{\star})q(x),$ 
	where $\kappa^3_{\star}$ is given by condition $(\textbf{\textit{H}}_{a_2})$ below, and
	
	\[
	h^{*}(x)=\begin{cases}
		\frac{N h(x)}{N-h(x)} &\text{if } \	h(x)<N ,\\
		+\infty &\text{if}	\ h(x)\geq N,
	\end{cases}
	\]
	for all $x\in\mathbb{R}^N$,  where  $\mathcal{H}:\mathbb{R}_{0}^{+}\to \left\lbrace 0,1\right\rbrace  $ is given by 	
	
	\[
	\mathcal{H}(\kappa^3_{\star})=\begin{cases}
		1 & \text{ if }     \   	\kappa^3_{\star}>0,\\
		0 & \text{ if }     \       \kappa^3_{\star}<0.
	\end{cases}
	\]	
	Furthermore, we assume that the set  $\mathcal{C}_{h}$ defined as $\lbrace x \in \mathbb{R}^N \mid r(x) = h^{*}(x) \rbrace$ is not empty.
	
	\noindent	
	The operators $\mathcal{A}_i: X \to \mathbb{R}$, where $i$ can be either 1 or 2, and the operator $\mathcal{B}: X \to \mathbb{R}$, are defined as follows
  $$\mathcal{A}_i(u)= a_i(| u|^{p(x)}) | u|^{p(x)-2} u \text{ and }\mathcal{B}(u)= \displaystyle\int_{\mathbb{R}^N}\dfrac{1}{p(x)}\Big(A_1(|\nabla u|^{p(x)})+ b(x)A_2(|u|^{p(x)})\Big)dx,$$
 where $X$ is the following Banach space
	$$X:= W^{1,p(x)}_{b}(\mathbb{R}^N)\cap W^{1,h(x)}_{b}(\mathbb{R}^N).$$	
	Function $A_i(.)$ is defined as $A_i(t)=\displaystyle\int_{0}^{t}a_i(k)d k$
	 and function $a_i(.)$ is from condition $(\textbf{\textit{H}}_{a_1})$ below.
	
	In this paper, we shall consider the function $a_i: \mathbb{R}^+ \to \mathbb{R}^+$, which satisfies the following assumptions for either $i=1$ or $i=2$:	
	
	\begin{itemize}
		\item[$(\textbf{\textit{H}}_{a_1})$] $a_i(.)$ is of   class $C^1$.		
		\item[$(\textbf{\textit{H}}_{a_2})$] There exist positive constants $\kappa_i^0, \kappa_i^1, \kappa_i^2$ and $\kappa^3_{\star,}$ for $i=1$ or $2$, such that 
		
		$$\kappa_i^0 + \mathcal{H}(\kappa^3_{\star})\kappa_i^2 t^{\frac{q(x)-p(x)}{p(x)}} \leq a_i(t ) \leq \kappa_i^1 + \kappa^3_{\star}t^{\frac{q(x)-p(x)}{p(x)}},
		\
		\hbox{for a.e.} 
		\
		x\in \mathbb{R}^N
		\		
		\hbox{and  all}
		 \
		t \geq 0.$$
		
		\item[$(\textbf{\textit{H}}_{a_3})$] There exists a positive constant $c>0$ such that 
		
		$$ \min \left\lbrace a_i(t^{p(x)})t^{p(x)-2}, 
		\
		a_i(t^{p(x)})t^{p(x)-2}
		+t \frac{\partial(a_i(t^{p(x)})t^{p(x)-2})}{\partial t } \right\rbrace \geq ct^{p(x)-2},
		\
		\hbox{for a.e.}
		\
		x\in \mathbb{R}^N
		\
		\hbox{and  all}
		\
		t>0.$$		
		\item[$(\textbf{\textit{H}}_{a_4})$]
		There exist positive constants $\gamma$, $\alpha_i$ (for $i=1$ or $2$), and a positive function $\vartheta$ satisfying condition $(\textbf{\textit{F}}_2)$ below, such that	
		$$
		A_i(t) \geq \frac{1}{\alpha_i}a_i(t)t \text{ with } h^+ < \vartheta(x) < r^- \text{ and } \frac{q^+}{p^+} \leq \frac{\alpha_i}{\gamma} < \frac{\vartheta^-}{p^+}, \text{ for all }t \geq 0,
		$$
	 where $\gamma_i$ satisfies condition $(\textbf{\textit{K}}_2)$ below.
	\end{itemize}
	
	Next, let $K: \mathbb{R}_0^+ \to \mathbb{R}^+$ be a nondecreasing and continuous 
	Kirchhoff function
	such that
	\begin{itemize}
		\item[$(\textbf{\textit{K}}_1)$] There exists $\mathfrak{K}_0>0$ such that 
		$$ K(t)\geq \mathfrak{K}_0 =K(0), 
		\ \hbox{for all} \ 
		t \in \mathbb{R}_{0}^+.$$
		
		\item[$(\textbf{\textit{K}}_2)$] There exists $\gamma \in (\frac{q^+}{r^-}, 1]$ such that 
		
		$$ \widehat{K} \left(t\right) \geq \gamma K(t)t, 
		\ \hbox{for all} \ 
		 t \in \mathbb{R}_{0}^+,
		 \
		 \hbox{where}
		 \
		 \widehat{K}\left(t\right):= \int_{0}^{t}K(s)ds.$$
	\end{itemize}	
	There are many functions satisfying conditions $ (\textbf{\textit{K}}_1)-(\textbf{\textit{K}}_2)$, for example  $K(t)=\mathfrak{K}_0 + \mathfrak{K}_1t^{\frac{1}{\gamma}},$ 
	for 
	$\gamma \leq 1$, $\mathfrak{K}_0>0,$ and $\mathfrak{K}_1\geq 0$.	
	
	 In 
	 recent years, increasing attention has been paid to the study of differential and partial differential equations involving variable exponent. The interest in studying such problems was stimulated by their many physical applications. For example,
	  they 
	  have been
	  applied in nonlinear elasticity problems, electrorheological fluids, image processing, flow in porous media,
	  and elsewhere, see e.g. Chen et al. \cite{C1}, Diening et al. \cite{Dien0,Dien1},  Halsey \cite{H1}, R\v{a}dulescu and Repov\v{s} \cite{radulescu1}, Ru\v{z}i\u{c}ka \cite{Ru1,Ru2}, and the references therein).\par
	
	We shall illustrate the degree of generality of the kind of problems studied here, with adequate hypotheses on functions $a_1$ and $a_2$,  by exhibiting some examples of problems which are also interesting from the mathematical point of view and have a wide range of applications in physics and other fields.
		
	\textbf{Example I.} Considering $a_1\equiv 1$ and $a_2\equiv 1$, we see that $a_1$ and $a_2$ satisfiy conditions  $(\textbf{\textit{H}}_{a_1}),(\textbf{\textit{H}}_{a_2})$ and $(\textbf{\textit{H}}_{a_3})$ for $\kappa^0_{i}=\kappa^1_{i}=1,$
	 $\kappa^2_{i}>0$ (where $i=1$ or $2$), and $\kappa^3_{\star}=0$. In this particular case, we are investigating the following problem:	
	
	\begin{eqnarray*}
		\label{ss1.1}
		\begin{cases}
			K\Big( \mathcal{B}(u)\Big)
		\Big(\Delta_{p(x)}u -b(x)| u|^{p(x)-2} u\Big) =|u|^{r(x)-2}u+\uplambda(x) \frac{ \partial \mathcal{F}}{\partial u}(x,u,v) & \text{in } \mathbb{R}^N, \\
			K\left(\mathcal{B}(v)\right)
			\Big(-\Delta_{p(x)}v+b(x)|v|^{p(x)-2} v\Big) =|v|^{r(x)-2}v+\uplambda(x) \frac{ \partial \mathcal{F}}{\partial v}(x,u,v) & \text{in } \mathbb{R}^N,
		\end{cases}
	\end{eqnarray*}	
where 	$$\mathcal{B}(u)=\displaystyle\int_{\mathbb{R}^N }\dfrac{1}{p(x)}\left(|\nabla u|^{p(x)}+ b(x)|u|^{p(x)} \right)dx$$ and 
$$\mathcal{B}(v)=\displaystyle\int_{\mathbb{R}^N }\dfrac{1}{p(x)}\left(|\nabla v|^{p(x)}+ b(x)|u|^{p(x)} \right)dx.$$	
	The operator $ \Delta_{p(x)}u:=\textrm{div}\,(|\nabla u|^{p(x)-2}\nabla u)$ is known as the $p(x)$-Laplacian,  which coincides with the usual $p$-Laplacian when $p(x)=p$, and with the Laplacian when $p(x)=2$.
	
	\textbf{Example II.} Considering the functions $$a_1(t)= 1+ t^{\frac{q(x)-p(x)}{p(x)}}
	\
	\hbox{ and}
	\
	a_2(t)= 1+ t^{\frac{q(x)-p(x)}{p(x)}},$$ 
  We observe that both $a_1$ and $a_2$ satisfy conditions 
	 $(\textbf{\textit{H}}_{a_1}),(\textbf{\textit{H}}_{a_2})$
	 and
	  $(\textbf{\textit{H}}_{a_3}),$ with $\kappa^0_{i}=\kappa^1_{i}=\kappa^2_{i}=\kappa^3_\star=1$ ( for $i=1$ or $2$). In this case, we are investigating the following noncooperative $p\&q$-Laplacian system:
	\begin{eqnarray*}
		\label{ss1.1}
		\begin{cases}
			K\left(\mathcal{B}(u)\right)
			\Big(\Delta_{p(x)}u +\Delta_{q(x)}u -b(x) (| u|^{p(x)-2} u +| u|^{q(x)-2} u) \Big)=|u|^{r(x)-2}u+\uplambda(x) \frac{ \partial \mathcal{F}}{\partial u}(x,u,v) & \text{in } \mathbb{R}^N, \\
			K\left(\mathcal{B}(v)\right)
			\Big(-\Delta_{p(x)}v -\Delta_{q(x)}v +b(x) (| v|^{p(x)-2} v +| v|^{q(x)-2} v )\Big) =|v|^{r(x)-2}v+\uplambda(x) \frac{ \partial \mathcal{F}}{\partial v}(x,u,v) & \text{in } \mathbb{R}^N,
		\end{cases}
	\end{eqnarray*}
	where $$\mathcal{B}(u)=\displaystyle\int_{\mathbb{R}^N }\Bigg(\dfrac{1}{p(x)}\left(|\nabla u|^{p(x)}+ b(x)|u|^{p(x)} \right)+\dfrac{1}{q(x)}\left(|\nabla u|^{q(x)}+ b(x)|u|^{q(x)} \right)\Bigg)dx$$
	and	
	$$\mathcal{B}(v)=\displaystyle\int_{\mathbb{R}^N }\Bigg(\dfrac{1}{p(x)}\left(|\nabla v|^{p(x)}+ b(x)|v|^{p(x)} \right)+\dfrac{1}{q(x)}\left(|\nabla v|^{q(x)}+ b(x)|v|^{q(x)} \right)\Bigg)dx.$$ 
	This class of systems arises in various applications, such as reaction-diffusion systems described by
	\begin{equation} \label{DCE}
		\frac{\partial u}{\partial t}= \textrm{div}\left(a_1(\nabla u)\nabla u \right)+ P(x,u),
		\
		\hbox{where }
		\
		 a_1(\nabla u)=|\nabla u|^{p(x)-2}+|\nabla u|^{q(x)-2},
	\end{equation}
  where the reaction term $P(x, u)$ is a polynomial of $u$ with variable coefficients. Such systems have wide applications in physics and related sciences, including plasma physics, biophysics, and chemical reaction design. In these applications, the function $u$ represents concentration, the first term on the right-hand side of \eqref{DCE} accounts for diffusion with a diffusion coefficient $ a_1(\nabla u)$, and the second term represents the reaction, which is related to source and loss processes, typically in chemical and biological applications. For further details, interested readers can refer to works by Mahshid and Razani \cite{Mahshid}, He and Li \cite{He}, and the references therein.  
	
	We continue with other examples that are also interesting from the mathematical point of view.	
	
	\textbf{Example III.} Considering $a_1(t)= 1+ \frac{t}{\sqrt{1+t^2}}$ and  $a_2\equiv 1$, we can observe that both $a_1$ and $a_2$ satisfy conditions  $(\textbf{\textit{H}}_{a_1}),(\textbf{\textit{H}}_{a_2})$ and $(\textbf{\textit{H}}_{a_3}),$ for $\kappa^0_{1}=\kappa^0_{2}=\kappa^1_{2}=1$,$\kappa^1_{1}=2$,  $\kappa^3_{\star}=0$, $\kappa^2_{1}>0,$ and $\kappa^2_{2}>0$. In this scenario, we are studying the following problem:
	\begin{eqnarray*}
		\label{ss1.1}
		\begin{cases}
			K\left(\mathcal{B}(u)\right)\Bigg(\textrm{div}\,\Big( \Big(1+ \frac{|\nabla u|^{p(x)}}{\sqrt{1+|\nabla u|^{2p(x)}}}\Big)|\nabla u|^{p(x)-2}\nabla u\Big) -b(x)| u|^{p(x)-2} u\Bigg) =|u|^{r(x)-2}u+\uplambda(x) \frac{ \partial \mathcal{F}}{\partial u}(x,u,v) ~~ \text{ in } \mathbb{R}^N, \\
			K\left(\mathcal{B}(v)\right)\Bigg(-\textrm{div}\,\Big( \Big(1+ \frac{|\nabla v|^{p(x)}}{\sqrt{1+|\nabla v|^{2p(x)}}}\Big)|\nabla v|^{p(x)-2}\nabla v\Big)+b(x)|v|^{p(x)-2} v \Bigg)=|v|^{r(x)-2}v+\uplambda(x) \frac{ \partial \mathcal{F}}{\partial v}(x,u,v) ~~ \text{ in } \mathbb{R}^N, 
		\end{cases}
	\end{eqnarray*}
	where $$\mathcal{B}(u)=\displaystyle\int_{\mathbb{R}^N }\frac{1}{p(x)}\Big(|\nabla u|^{p(x)} + \sqrt{1+|\nabla u|^{2p(x)}}+ b(x)|u|^{p(x)}\Big)dx$$ 
	and
	$$\mathcal{B}(v)=\displaystyle\int_{\mathbb{R}^N }\frac{1}{p(x)}\Big(|\nabla v|^{p(x)} + \sqrt{1+|\nabla v|^{2p(x)}}+ b(x)|v|^{p(x)}\Big)dx.$$
	The operator  $$\textrm{div}\,\Big( \Big(1+ \frac{|\nabla u|^{p(x)}}{\sqrt{1+|\nabla u|^{2p(x)}}}\Big)|\nabla u|^{p(x)-2}\nabla u\Big)$$
	is known as the $p(x)$-Laplacian-like operator or the generalized capillary operator, which has applications in various fields such as industry, biomedicine, and pharmaceuticals. For further details, you can refer to Ni and Serrin \cite{Ni}.	
	
	\textbf{Example IV.} Considering $a_1(t)= 1+ t^{\frac{q(x)-p(x)}{p(x)}} +\frac{1}{(1+t)^{\frac{p(x)-2}{p(x)}}}$ and $a_2(t)= 1+ t^{\frac{q(x)-p(x)}{p(x)}}$, we have that $a_1$ and $a_2$  satisfy conditions
	  $(\textbf{\textit{H}}_{a_1}),(\textbf{\textit{H}}_{a_2})$ and $(\textbf{\textit{H}}_{a_3})$ with $\kappa^0_{1}=\kappa^0_{2}=\kappa^1_{2}=1$, $\kappa^1_{1}= 2$ and $\kappa^3_{\star}=\kappa^2_{1}=\kappa^2_{2}=1$. In this case we are studying problem
	
	\begingroup\makeatletter\def \f@size{10}\check@mathfonts
	\begin{eqnarray*}
		\label{ss1.1}
		\begin{cases}
				\begin{gathered}
			K\left(\mathcal{B}(u)\right) \Bigg(\Delta_{p(x)}u +\Delta_{q(x)}u +\textrm{div}\,\Big( \dfrac{|\nabla u|^{p(x)-2}\nabla u}{(1+ |\nabla u|^{p(x)})^{\frac{p(x)-2}{p(x)}}}\Big)-b(x) \Big(| u|^{p(x)-2} u +| u|^{q(x)-2} u \Big)\Bigg)\\
			=|u|^{r(x)-2}u+\uplambda(x) \frac{ \partial \mathcal{F}}{\partial u}(x,u,v) ~~  \text{ in } \mathbb{R}^N, \\
			K\left(\mathcal{B}(v)\right)\Bigg(-\Delta_{p(x)}v -\Delta_{q(x)}v- \textrm{div}\,\Big( \dfrac{|\nabla v|^{p(x)-2}\nabla v}{(1+ |\nabla v|^{p(x)})^{\frac{p(x)-2}{p(x)}}}\Big) +b(x) \Big(| v|^{p(x)-2} v +| v|^{q(x)-2} v \Big)\Bigg) \\
			=|v|^{r(x)-2}v+\uplambda(x) \frac{ \partial \mathcal{F}}{\partial v}(x,u,v)  ~~\text{ in } \mathbb{R}^N .
			\end{gathered}
		\end{cases}
	\end{eqnarray*}
	\endgroup
		where 
		$$\mathcal{B}(u)=\displaystyle\int_{\mathbb{R}^N }\Bigg(\dfrac{1}{p(x)}\left(|\nabla u|^{p(x)}+ b(x)|u|^{p(x)} \right)+\dfrac{1}{q(x)}\left(|\nabla u|^{q(x)}+ b(x)|u|^{q(x)} \right) + \frac{1}{2}(1+ |\nabla u|^{p(x)})^{\frac{2}{p(x)}}\Bigg)dx$$ 
		and
		 $$\mathcal{B}(v)=\displaystyle\int_{\mathbb{R}^N }\Bigg(\dfrac{1}{p(x)}\left(|\nabla v|^{p(x)}+ b(x)|v|^{p(x)} \right)+\dfrac{1}{q(x)}\left(|\nabla v|^{q(x)}+ b(x)|v|^{q(x)} \right) + \frac{1}{2}(1+ |\nabla v|^{p(x)})^{\frac{2}{p(x)}}\Bigg)dx.$$
		 Moreover, the class of systems \eqref{s1.1} can include either a single model of the divergence operators mentioned above, as in Examples I–IV, or two different models in each equation for divergence operators simultaneously, depending on the studied phenomenon. Also, every equation in this class can be either degenerate or non-degenerate.\par
    	In the case of a single equation, system \eqref{s1.1} is related to a model that was first proposed by Kirchhoff in 1883. This model represents the stationary version of the Kirchhoff equation, which can be written as:	
	\begin{equation}
		\label{e1.1}
		\uprho\frac{\partial^2u}{\partial t^2}-\left( \dfrac{\uprho_0}{h}+\dfrac{E}{2L}\int_{0}^{L}\left\vert\dfrac{\partial u(x)}{\partial x}\right\vert^2dx \right)\dfrac{\partial^2u}{\partial x^2}=0.
	\end{equation}
	\noindent
   This equation extends the classical D'Alembert wave equation by considering the small vertical vibrations of a stretched elastic string with variable tension and fixed ends. One distinctive feature of equation (\ref{e1.1}) is the presence of a nonlocal coefficient:
	 $$\dfrac{\uprho_0}{h}+\dfrac{E}{2L}{\displaystyle\int_{0}^{L}\left\vert\dfrac{\partial u}{\partial x}\right\vert^2dx}.$$ This coefficient depends on the average value:
	 $$\dfrac{1}{2L}{\displaystyle\int_{0}^{L}\left\vert\dfrac{\partial u}{\partial x}\right\vert^2dx}.$$
	 As a result, the equation is no longer a pointwise equation, and this nonlocal aspect distinguishes it from the classical wave equation.
	  
	  \noindent
	  The parameters in equation (\ref{e1.1}) have the following meanings: $u=u(x,t)$ represents the transverse string displacement at the spatial coordinate $x$ and time $t$, $E$ is the Young modulus of the material, also known as the elastic modulus, which measures the string's resistance to elastic deformation, $\uprho$ is the mass density, $L$ is the length of the string, $h$ is the area of cross-section, and $\uprho_0$ is the initial tension (for more details see  Kirchhoff \cite{Kir}). 
	  
	  Almost one century later, Jacques-Louis Lions \cite{Lions} returned to the equation and proposed a general Kirchhoff equation in arbitrary dimension with external force term which was written as
	
	\begin{eqnarray}
		\label{s1.3}
		\begin{cases}
			\frac{\partial^2u}{\partial t^2}-\left(
			a+b \int_{\Omega  }\left| \nabla u\right| ^2dx \right)\Delta u=f(x,u) &\quad \text{ in }\Omega, \\
			u=0 &\quad\text{on }\partial\Omega;
		\end{cases}
	\end{eqnarray}
this problem is often referred to as a nonlocal problem because it involves an integral over the domain $\Omega$. This integral component introduces mathematical complexities that make the study of such problems particularly interesting. The nonlocal problem serves as a model for various physical and biological systems in which the variable $u$ represents a process dependent on its own average, such as population density. For further references on this subject, interested readers can explore the works of Arosio and Pannizi \cite{Arosio}, Cavalcanti et al. \cite{Cavalcanti},  Chipot and Lovat \cite{Chipo}, and  Corr\^{e}a and Nascimento \cite{Corr2}, along with the references provided therein.\par

On one hand, it is widely acknowledged that the class of elliptic problems with constant critical exponents in bounded or unbounded domains holds a significant place in the literature. This class of problems was first introduced in the seminal paper by Brezis and Nirenberg \cite{Brezis}, which primarily focused on Laplacian equations. Subsequently, various extensions of the results presented in \cite{Brezis} have been explored in many directions. A notable feature of elliptic equations involving critical growth is the issue of a lack of compactness, which is closely tied to the variational approach. To address this lack of compactness, P.L. Lions \cite{Lions1} developed a method employing the concentration compactness principle (CCP) to establish that a minimizing sequence or a Palais-Smale (PS) sequence is precompact. Following this development, a variable exponent version of P.L. Lions' concentration-compactness principle for bounded domains was independently formulated by Bonder and Silva \cite{Bonder}, Fu \cite{Fu}, while the version for unbounded domains was introduced by Fu \cite{Fu1}. Subsequently, numerous researchers have employed these results to investigate critical elliptic problems involving variable exponents, as evidenced by the works of Alves et al. \cite{Alves0,Alves}, Chems Eddine et al. \cite{chems1,chems2,chems3}, Hurtado et al. \cite{Hurtado}, Liang et al. \cite{Liang01,Liang0,Liang1,Liang2}, Fu and Zhang \cite{Zhang}.\par

On the other hand, over the past few decades, there has been significant interest among researchers in studying elliptic problems that lead to indefinite functionals. For instance, in the work by Benci \cite{Benci}, it was assumed that $X$ is a Hilbert space, and $f$ satisfies the Palais-Smale condition, and has the form
$$ f(u)= \frac{1}{2}\left\langle L(u),u\right\rangle +\Phi(u),\quad u \in H,$$
where $L$ is bounded self-adjoint operator and $\Phi'$ is compact. Nevertheless, the solution spaces are not necessarily Hilbert spaces. To overcome this difficulty,  in \cite{Li1}, Li introduced a Limit index theory and applied it
to estimate the number of solutions for the following noncooperative $p$-Laplacian elliptic system with Dirichlet boundary conditions
\begin{eqnarray}
	\label{LI1.1}
	\begin{cases}
		\Delta_{p}u =\frac{\partial F}{\partial u}(x,u,v) & \text{in } \Omega , \\
		-\Delta_{p}v =\frac{\partial F}{\partial v}(x,u,v) & \text{in } \Omega,\\
		u=0,\quad v=0,& \text{on  } \partial\Omega  .
	\end{cases}
\end{eqnarray}

Following that, Huang and Li \cite{Huang} studied the following noncooperative $p$-Laplacian elliptic system in the unbounded domain of $\mathbb{R}^N$ by using  the principle of symmetric criticality and the Limit index theory
\begin{eqnarray*}
	\label{ss1.1}
	\begin{cases}
		\Delta_{p}u -| u|^{p-2} u = \frac{\partial F}{\partial u}(x,u,v) & \text{in } \mathbb{R}^N, \\
		-\Delta_{p}v+|v|^{p-2} v = \frac{\partial F}{\partial v}(x,u,v) & \text{in } \mathbb{R}^N ,\\
		u,v \in W^{1,p}(\mathbb{R}^N),
	\end{cases}
\end{eqnarray*}
where $1<p<N,$ and 
extended some results of Li \cite{Li1}. 
Next, Cai and Li \cite{Cai}  dealt with the case when the corresponding functional of (\ref{LI1.1}) may
not be locally Lipschitz continuous in Banach spaces.
Lin and Li \cite{Lin1}, studied problem (\ref{LI1.1}) with critical exponents of the form
\begin{eqnarray}
	\label{lii1.1}
	\begin{cases}
		\Delta_{p}u = | u|^{p^*-2}u +\frac{\partial F}{\partial u}(x,u,v) & \text{in } \Omega , \\
		-\Delta_{p}v = |v|^{p^*-2} v+ \frac{\partial F}{\partial v}(x,u,v) & \text{in } \Omega  ,\\
		u=v=0 & \text{on  } \partial\Omega.
	\end{cases}
\end{eqnarray}
where $\Omega$ is a bounded domain in $\mathbb{R}^N$, $1<p,q<N$, $p^*=\frac{Np}{N-p}$ and $q^*=\frac{Nq}{N-q}$,
and established the existence of multiple solutions for problem (\ref{lii1.1}) without using Concentration compactness principle. Some similar results for the noncooperative $p(x)$-Laplacian elliptic problems were obtained by Liang et al. \cite{Liang1,Liang2}. Recently, Chems Eddine \cite{chems34} extended these results to the problem \eqref{s1.1} when $K\equiv1$, $\uplambda$ is a real number, and the functions $p$ and $q$ are Lipschitz continuous.\par 

Our objective in this paper, is to study the existence and multiplicity of solutions for a class of the generalized noncooperative Schr\"{o}dinger- Kirchhoff-type systems with critical nonlinearity involving a general variable exponent elliptic operator in $\mathbb{R}^N$.  More precisely, our main results of this work extend, complement, and complete several works, in particular Chems Eddine \cite{chems34}, 
Fang and Zhang \cite{Fang1},
Huang and Li \cite{Huang}, Li \cite{Li1}, Liang and Zhang \cite{Liang2},  and some papers listed therein. 

As we shall see in the next sections, there are three main difficulties in our situation. First, the energy functional corresponding to problem (\ref{s1.1}) is strongly indefinite. Here, we mean strongly indefinitely that a functional is unbounded from below and from above on any subspace of finite codimension. Hence, we cannot apply the Mountain pass theorem for the energy functional. The second difficulty in solving problem (\ref{s1.1}) is the lack of compactness which can be illustrated by the fact that the embedding of $W^{1,h(x)}_{b}(\mathbb{R}^N) \hookrightarrow L^{h^{*}(x)}(\mathbb{R}^N)$
is no longer compact. The third
difficulty is that 
 problem (\ref{s1.1}) involves nonlocal terms $K\left(\mathcal{B}(u)\right)$ and  $K\left(\mathcal{B}(v)\right)$  which prevent us from applying the methods as before. To overcome these difficulties, we use the limit index theory developed by Li \cite{Li1}, the  principle of concentration compactness (Theorem \ref{ccpp}), the
 concentration compactness principle at infinity  (Theorem \ref{ccppp}) for the weighted-variable exponent Sobolev spaces $W^{1,h(x)}_{b}(\mathbb{R}^N)$, and the principle of symmetric criticality of Krawcewicz and Marzantowicz \cite{Krawcewicz}.\par 

Throughout this paper, we shall assume that $\mathcal{F}$ satisfies the following conditions:
\begin{itemize}
	\item[$(\textbf{\textit{F}}_1)$] $\mathcal{F}\in C^1(\mathbb{R}^N\times\mathbb{R}^2,\mathbb{R}^+)$ and
	it 
	 satisfies
	$$ \left| \frac{ \partial \mathcal{F}}{\partial \upeta}(x,\upeta,\upxi )\right| +\left| \frac{ \partial \mathcal{F}}{\partial \upxi} (x,\upeta,\upxi )\right| \leq f_1(x)\left| \upeta\right|^{\ell(x)-1}+ f_2(x)\left| \upxi \right|^{\ell(x)-1},
	$$
	where $\ell\in \mathcal{M}(\mathbb{R}^N)$, 
	$q^+<\ell(x)< h^{*}(x)  
	\
	\hbox{for all}
	\
	 x \in \mathbb{R}^N ,$ and $0\leq f_1, f_2\in L^{l(x)}\cap L^{\infty}(\mathbb{R}^N) \text{ with }  l(x)= h^{*}(x)/ (h^{*}(x)-\ell(x)).$
	
	\item[$(\textbf{\textit{F}}_2)$] There exists $\vartheta(x) $ such that $h^+< \vartheta(x) $ and  $0<\vartheta(x)\mathcal{F}(x,\upeta,\upxi )\leq \upeta\frac{ \partial \mathcal{F}}{\partial \upeta}(x,\upeta,\upxi )+ \upxi \frac{ \partial \mathcal{F}}{\partial \upxi} (x,\upeta,\upxi )$, for all $(x,\upeta,\upxi )\in  (\mathbb{R}^N\times\mathbb{R}^2)$.	
	
	\item[$(\textbf{\textit{F}}_3)$]
	$\upeta\frac{ \partial \mathcal{F}}{\partial \upeta}(x,\upeta,\upxi )\geq  0$ for all $(x,\upeta,\upxi )\in\mathbb{R}^N\times\mathbb{R}^2$.	
	
	\item[$(\textbf{\textit{F}}_4)$]  $\mathcal{F}$ is even in $(\upeta,\upxi )$ : $ \mathcal{F}(x,\upeta,\upxi )=\mathcal{F}(x,-\upeta,-\upxi )$ for all $(x,\upeta,\upxi )\in\mathbb{R}^N\times\mathbb{R}^2$.
	
	\item[$(\textbf{\textit{F}}_5)$]$ \mathcal{F}(x,\upeta,\upxi )=\mathcal{F}(|x|,\upeta,\upxi )$ for all $(x,\upeta,\upxi )\in\mathbb{R}^N\times\mathbb{R}^2$.	
\end{itemize}

The main result of this paper is as follows.

\begin{theorem}\label{theooo}
Assume that conditions $(\textbf{\textit{H}}_{a_1})-(\textbf{\textit{H}}_{a_4})$, $(\textbf{\textit{K}}_1)-(\textbf{\textit{K}}_2),$ and $(\textbf{\textit{F}}_1)-(\textbf{\textit{F}}_5)$ are satisfied. Then there exists a constant $\uplambda_{\star} >0$, such that if 
$\uplambda(x)$ satisfies the following condition 
	$$0<\uplambda^-:=\inf_{ x\in\mathbb{R}^N}\uplambda(x)\leq \uplambda^+:= \| \uplambda\|_{L^{\infty}(\mathbb{R}^N)}\leq \uplambda_{\star},$$
	 then  problem (\ref{s1.1}) possesses infinitely many weak solutions in $ X\times X$.
\end{theorem}

The  paper is organized as follows. In Section \ref{subsect21}, we briefly present some properties of the generalized weighted Sobolev spaces with variable exponents. In addition, we introduce the principle of concentration compactness and the concentration compactness principle at infinity in the generalized weighted-variable exponent Sobolev spaces. In Section \ref{subsect22}, we mainly introduce the limit index theory due to Li \cite{Li1}.
 In Section \ref{sect3}, we provide proof for the main results,  after we have verified the Palais-Smale condition at some special energy levels, by using the concentration-compactness principle.	
	
	\section{Preliminaries and basic notations}\label{sec2}
	In this section, we introduce some definitions and results which will be used in the next section. Throughout this paper, we employ the following notation and conventions: We use $\rightarrow$ to denote strong convergence, $\rightharpoonup$ for weak convergence, and $\overset{\ast}{\rightharpoonup}$ for weak-* convergence. For any given $\rho>0$ and $x\in \Omega$, $B_{\rho}(x)$ represents the ball with radius $\rho$ centered at $x$. The duality pairing between $X'$ and $X$ is represented by $\langle \cdot,\cdot \rangle$.  $C$ and $c$ denote a positive constants and can be determined based on specific conditions.
	
	\subsection{Generalized weighted  variable Sobolev spaces and the  principle of concentration compactness}\label{subsect21}
     First, we shall introduce some fundamental results from the theory of Lebesgue-Sobolev spaces with variable exponents. The details can be found in Diening et al. \cite{Dien1}, Fan and Zhao \cite{Fan3},  and 
     Kov\'{a}\v{c}ik and R\'{a}kosn\'{i}k \cite{Kov1}. Let $\mathcal{M}(\mathbb{R}^N)$ be the set of all measurable real functions on $\mathbb{R}^N$. We define
	\begin{equation*}
		C_{+}(\mathbb{R}^N)= \big\{  p \in C(\mathbb{R}^N) :  \text{ess inf}_{x\in \mathbb{R}^N}p(x)> 1 \big\} .
	\end{equation*}	
  	Additionally, we denote by $C_{+}^{\text{log}}(\mathbb{R}^N)$  the set of functions $p\in C_{+}(\mathbb{R}^N)$ that satisfy the log-Holder continuity condition
$$\sup \left\{ |p(x)-p(y)| \log \frac{1}{|x-y|} : x,y\in \mathbb{R}^N, 0< |x-y|<\frac{1}{2} \right\} < \infty.$$

	For any $ p \in C_{+}(\mathbb{R}^N)$ we define
	\begin{gather*}
		p^{-}:= \text{ess inf}_{x\in \mathbb{R}^N}p(x) \quad\text{ and} ~p^{+}:= \text{ess sup}_{x\in \mathbb{R}^N}p(x) .
	\end{gather*}
	For any $p \in C_{+}(\mathbb{R}^N)$, we define the variable exponent Lebesgue space as	
	\begin{gather*}
		L^{p(x)}(\mathbb{R}^N)= \left\{ u \in \mathcal{M}(\mathbb{R}^N) : \int_{\mathbb{R}^N}|u(x)|^{p(x)}dx <\infty \right\},
	\end{gather*}	
	endowed with the Luxemburg norm	
	\begin{equation*}
		|u|_{p(x)}:=|u|_{L^{p(x)}(\mathbb{R}^N)} = \inf \left\{\tau >0 : \int_{\mathbb{R}^N }\left|\frac{u(x)}{\tau}\right|^{p(x)} dx \leq 1 \right \}.
	\end{equation*}	
	Let $b \in \mathcal{M}(\mathbb{R}^N)$, and $b(x)>0$ for a.e $x\in \mathbb{R}^N$. Define the weighted variable exponent Lebesgue space $L_{b}^{p(x)}(\mathbb{R}^N)$ by 
	\begin{gather*}
		L_{b}^{p(x)}(\mathbb{R}^N)= \left\{ u \in \mathcal{M}(\mathbb{R}^N) : \int_{\mathbb{R}^N}b(x)|u(x)|^{p(x)}dx <\infty \right\},
	\end{gather*}
	with  the norm 	
	\begin{equation*}
		|u|_{b,p}:=|u|_{L_{b}^{p(x)}(\mathbb{R}^N)} = \inf \left\{ \tau >0 : \int_{\mathbb{R}^N }b(x)\left|\frac{u(x)}{\tau}\right|^{p(x)} dx \leq 1 \right \}.
	\end{equation*}
	
	From now on, we shall assume that $w \in L^{\infty}(\mathbb{R}^N)$ with $b_0:= \rm{ess} \inf_{x\in \mathbb{R}^N} b(x) > 0$. Then obviously $L_{b}^{p(x)}(\mathbb{R}^N)$ is a Banach space (see Cruz-Uribe et al. \cite{Cruz1} for details) , and the norms $|u|_{b,p}$ and $|u|_{p}$ are equivalent in $L_{b}^{p}(\mathbb{R}^N)$. \par 	
	On the other hand, the variable exponent Sobolev space $W^{1,p(x)}(\mathbb{R}^N)$ is defined by
	\[
	W^{1,p(x)}(\mathbb{R}^N)=\{ u\in L^{p(x)}(\mathbb{R}^N):| \nabla
	u| \in L^{p(x)}(\mathbb{R}^N)\},
	\]
	and is endowed with the norm
	\[
	\| u\| _{1,p(x)}:=\| u\|_{W^{1,p(x)}(\mathbb{R}^N)}
	=|u| _{p(x)}+| \nabla u| _{p(x)},   ~~~~ \  \hbox{for all} \  u\in W^{1,p(x)}(\mathbb{R}^N ).
	\]	
   Next, we define the weighted-variable exponent Sobolev space $W_{b}^{1,p(x)}(\mathbb{R}^N)$ as follows
	\[
	W_{b}^{1,p(x)}(\mathbb{R}^N)=\big\{ u\in L_{b}^{p(x)}(\mathbb{R}^N):| \nabla
	u| \in L_{b}^{p(x)}(\mathbb{R}^N)\big\}.
	\]
	This space is equipped with the norm
	\begin{equation*}
		\| u\| _{b,p} := \inf \left\{ \tau >0 :  \int_{\mathbb{R}^N } \left|\frac{\nabla u(x)}{\tau}\right|^{p(x)} +b(x)\left|\frac{u(x)}{\tau}\right|^{p(x)} dx \leq 1 \right \},  \  \hbox{for all} \  u \in W_{b}^{1,p(x)}(\mathbb{R}^N).
	\end{equation*}	
     It is worth noting that the norms $\| u\|_{b,p}$ and $\| u\| _{1,p(x)}$ are equivalent in the space $W_{b}^{1,p(x)}(\mathbb{R}^N)$. Moreover, if $p^->1$, then the spaces $L^{p(x)}(\mathbb{R}^N)$, $W^{1,p(x)}(\mathbb{R}^N)$, and $W_{b}^{1,p(x)}(\mathbb{R}^N)$ are separable, reflexive, and uniformly convex Banach spaces.
     
     We now present some essential facts that will be utilized later.
     \begin{proposition}[see \cite{Dien1, Fan3}] \label{prop1}
     	The conjugate space of $L^{p(x)}(\mathbb{R}^N)$ is $L^{p'(x)}(\mathbb{R}^N)$, where
     	\[
     	\frac{1}{p(x)}+\frac{1}{p'(x)}=1.
     	\]
     	Furthermore, for any $(u,v)\in L^{p(x)}(\mathbb{R}^N)\times L^{p'(x)}(\mathbb{R}^N)$,
     	we have the following H\"older-type inequality
     	\[
     	| \int_{\mathbb{R}^N}uvdx|
     	\leq (\frac{1}{p^{-}}+ \frac{1}{(p')^{-}})| u|_{p(x)}| v| _{p'(x)}
     	\leq 2| u| _{p(x)}| v|_{p'(x)}.
     	\]
     \end{proposition}
	
	\begin{proposition}[see \cite{Dien1, Fan3}] \label{prop2}
		Denote $\rho_p (u)=\int_{\mathbb{R}^N}| u| ^{p(x)}dx$, for all $u\in L^{p(x)}(\mathbb{R}^N)$. We have
		\[
		\min \{ | u| _{p(x)}^{p^{-}},| u| _{p(x)}^{p^{+}}\}
		\leq \rho_p(u)\leq \max \{ | u| _{p(x)}^{p^{-}},| u| _{p(x)
		}^{p^{+}}\},
		\]
		
		and the following implications are true:  
		\begin{itemize}
			\item[(i)]  $| u| _{p(x)}<1$ (resp. $=1, >1$) $\Leftrightarrow \rho_p(u)<1$
			(resp. $=1,>1$),
			
			\item[(ii)] $| u| _{p(x)}>1 \Rightarrow | u| _{p(x)}^{p^{-}}\leq \rho_p(u)
			\leq | u| _{p(x)}^{p^{+}}$,
			
			\item[(iii)] $| u| _{p(x) }<1\Rightarrow | u| _{p(x)}^{p^{+}}\leq \rho_p(u)
			\leq | u| _{p(x)}^{p^{-}}$.
			
		\end{itemize}
	Additionally, in particular, for any sequence $\left\lbrace u_n\right\rbrace \subset L^{p(x)}(\mathbb{R}^N)$,
		$$| u_n|_{p(x)} \to 0 \text{ if and only if } \rho_p(u_n)\to ,0$$
		and 
		$$\left\lbrace u_n\right\rbrace \text{ is bounded in } L^{p(x)}(\mathbb{R}^N)  \text{ if and only if } \rho_p(u_n) \text{ is bounded in } \mathbb{R}.
		$$
	\end{proposition}
	According to Proposition \ref{prop2}, we can derive the following inequalities:
	\begin{equation}
		\label{inq21}
		\| u\| _{b,p}^{p^{-}}  \leq  \int_{\mathbb{R}^N }\left( \left|\nabla u(x)\right|^{p(x)} +b(x)\left|u(x)\right|^{p(x)} \right)dx \leq  \| u\| _{b,p}^{p^{+}}    \quad \text{ for} ~~~~ \| u\| _{b,p} \geq 1.
	\end{equation}
	\begin{equation}
		\label{inq22}
		\| u\| _{b,p}^{p^{+}}  \leq  \int_{\mathbb{R}^N }\left( \left|\nabla u(x)\right|^{p(x)} +b(x)\left|u(x)\right|^{p(x)} \right)dx \leq  \| u\| _{b,p}^{p^{-}}    \quad \text{ for} ~~~~ \| u\| _{b,p} \leq 1.
	\end{equation}	
Moreover, for any sequence $\left\lbrace u_n\right\rbrace\subset W_{b}^{1,p(x)}(\mathbb{R}^N)$:
 $$\text{When }	\| u_n\| _{b,p} \to 0, \text{  it is equivalent to } \int_{\mathbb{R}^N }\left( \left|\nabla u_n(x)\right|^{p(x)} +b(x)\left|u_n(x)\right|^{p(x)} \right)dx\to 0,$$
	and 
	$$\text{When } \left\lbrace u_n\right\rbrace \text{ is bounded in } W_{b}^{1,p(x)}(\mathbb{R}^N), \text{  it is equivalent to } \int_{\mathbb{R}^N }\left( \left|\nabla u_n\right|^{p(x)} +b(x)\left|u_n\right|^{p(x)} \right)dx \text{ is bounded in } \mathbb{R}.
	$$	
	\begin{proposition}[see \cite{Edmu1}] \label{prop3}
		Let $p$ and $q$ be measurable functions such
		that $p\in L^{\infty }(\mathbb{R}^N)$ and
		$1\leq p(x), q(x)\leq \infty $ for a.e.
		$x \in \mathbb{R}^N$. If $u\in L^{q(x)}(\mathbb{R}^N)$, $u\neq 0$, then the following inequalities hold:
		\begin{gather*}
		 \text{If }	~~ | u| _{p(x)q(x)}\leq 1, ~~\text{  then } ~~|u| _{p(x)q(x)}^{p^{-}}
			\leq \big|| u| ^{p(x)}\big| _{q(x)}\leq | u| _{p(x)q(x)}^{p^{+}},
			\\
		 \text{If }	~~ 	| u| _{p(x)q(x)}\geq 1, ~~\text{  then } ~~ |u| _{p(x)q(x)}^{p^{+}}
			\leq \big| | u| ^{p(x)}\big| _{q(x)}\leq | u| _{p(x)q(x)}^{p^{-}}.
		\end{gather*}
		In particular, if $p(x)=p$ is constant, then
		$| | u| ^{p}| _{q(x)}
		=|u| _{pq(x)}^{p}.$
	\end{proposition}	
	
	\begin{proposition}[see \cite{Dien1,Edmu1}] \label{prop5}
		Let $p\in  C_{+}^{\text{log}}(\mathbb{R}^N) $ be such that $p^+<N$ and let  $q\in C(\mathbb{R}^N)$ satisfy $1<q(x)<p^{*}(x)$ for each $x\in \mathbb{R}^N$, then there exists a continuous and compact embedding $W_{b}^{1,p(x)}(\mathbb{R}^N)\hookrightarrow L^{q(x)}(\mathbb{R}^N )$.
	\end{proposition}	
	
	\begin{proposition}[see \cite{Dien1,Edmu1}] \label{prop4}
		Let $p\in C_{+}^{\text{log}}(\mathbb{R}^N)$. Then, there exists a positive constant $ C^{\ast}$ such that
		\[
		| u| _{p^{\ast }(x)}\leq C^{\ast}\|u\| _{b,p},   \quad \  \hbox{for all} \ 
		u\in W_{b}^{1,p(x)}(\mathbb{R}^N).
		\]
	\end{proposition}	
	
	In the upcoming discussions, we shall work with the product space denoted as
	\[
	Y:=\Big(W^{1,p(x)}_{b}(\mathbb{R}^N)\cap W^{1,h(x)}_{b}(\mathbb{R}^N)\Big)\times\Big(W^{1,p(x)}_{b}(\mathbb{R}^N)\cap W^{1,h(x)}_{b}(\mathbb{R}^N)\Big).
	\]
   This space is endowed with the norm
	\[
	\| (u,v)\|_{Y}:=\max \left\lbrace  \|  u\|_{b,h },\|  v\|_{b,h} \right\rbrace ,~~~~~ \  \hbox{for all} \  (u,v)\in Y,
	\]
	where $\| u\|_{b,h}:= \| u\|_{b,p}+ \mathcal{H}(\kappa_{\star}^3)\| u\|_{b,q}$   represents the norm of  $W^{1,p(x)}_{b}(\mathbb{R}^N)\cap W^{1,h(x)}_{b}(\mathbb{R}^N)$. The space $Y^{*}$ corresponds to the dual space of $Y$ and is endowed with the standard dual norm.	
	
	\begin{definition}
		Consider a Banach space $Y$. An element $(u,v) \in Y$ is said to be a weak solution of the system  \eqref{s1.1} if 	
		\begin{align*}
			- K\left(
			\mathcal{B}(u)\right)\int_{\mathbb{R}^N }\Big( \mathcal{A}_1(\nabla u) .\nabla\tilde u 
			+ b(x)\mathcal{A}_2(u)\tilde u\Big)dx - \int_{\mathbb{R}^N  }|u|^{r(x)-2}u\tilde u\,dx  \\
			+K\left( \mathcal{B}(v)\right)\int_{\mathbb{R}^N }\Big( \mathcal{A}_1(\nabla v) .\nabla\tilde v  +b(x)\mathcal{A}_2(v)\tilde v \Big)dx - \int_{\mathbb{R}^N }|v|^{r(x)-2}v\tilde v\,dx  \\
			-\displaystyle\int_{\mathbb{R}^N }\uplambda(x) \frac{ \partial \mathcal{F}}{\partial u}(x,u,v)\tilde{u}\,dx
			-	\displaystyle\int_{\mathbb{R}^N }\uplambda(x) \frac{ \partial \mathcal{F}}{\partial v}(x,u,v)\tilde{v}\,dx  =0,	\end{align*}
		for all $(\tilde{u},\tilde{v})\in Y= \Big(W^{1,p(x)}_{b}(\mathbb{R}^N)\cap W^{1,h(x)}_{b}(\mathbb{R}^N)\Big)\times\Big(W^{1,p(x)}_{b}(\mathbb{R}^N)\cap W^{1,h(x)}_{b}(\mathbb{R}^N)\Big)$.
	\end{definition}
	
	The energy functional  $\tilde{E}_{\uplambda}:Y\longrightarrow\mathbb{R} $ associated with problem \eqref{s1.1} is given by, 
	\begin{align*}
		\tilde{E}_{\uplambda}(u,v)&=-\widehat{K}\left(
		\mathcal{B}(u(x))\right)+\widehat{K}\left(
		\mathcal{B}(v(x))\right)-\displaystyle\int_{\mathbb{R}^N }\dfrac{1}{r(x)}|u|^{r(x)}dx-\displaystyle\int_{\mathbb{R}^N }\dfrac{1}{r(x)}|v|^{r(x)}dx
		 -\int_{\mathbb{R}^N}\uplambda(x) \mathcal{F}(x,u,v)dx,
	\end{align*}
	for each $(u,v)$ in $Y$.\par	
	
 Through standard calculus, one can establish that, under the above assumptions, the energy functional $\tilde{E}_{\uplambda}:Y\to \mathbb{R}^N$ associated with problem (\ref{s1.1}) is well-defined and belongs to $C^1(Y,\mathbb{R})$. Its derivative, denoted as $\tilde{E}_{\uplambda}'(u,v)$, satisfies
	\begin{align*}
		\left\langle \tilde{E}_{\uplambda}'(u,v),(\tilde{u},\tilde{v})\right\rangle &= 	- K\left( \mathcal{B}(u)\right)\int_{\mathbb{R}^N } \Big( \mathcal{A}_1(\nabla u) .\nabla\tilde u  +  b(x)\mathcal{A}_2(u)\tilde u\Big)dx - \int_{\mathbb{R}^N  }|u|^{r(x)-2}u\tilde u\,dx \\
		&\quad	+ K\left(
		\mathcal{B}(v)\right)\int_{\mathbb{R}^N } \Big(\mathcal{A}_1(\nabla v) .\nabla\tilde v 
		+  b(x)\mathcal{A}_2(v)\tilde v\Big)dx - \int_{\mathbb{R}^N  }|v|^{r(x)-2}v\tilde v\,dx \\
		&\quad  -\displaystyle\int_{\mathbb{R}^N }\uplambda(x) \frac{ \partial \mathcal{F}}{\partial u}(x,u,v)\tilde{u}\,dx-\displaystyle\int_{\mathbb{R}^N }\uplambda(x) \frac{ \partial \mathcal{F}}{\partial v}(x,u,v)\tilde{v}\,dx,
	\end{align*}
   for all $(\tilde{u},\tilde{v})\in Y$. Consequently, the critical points of the functional $\tilde{E}_{\uplambda}$ correspond to weak solutions of the system \eqref{s1.1}. \par
	
	To establish our existence result, we need to address the loss of compactness in the inclusion $W_{b}^{1,p(x)}(\mathbb{R}^N)\hookrightarrow L^{p^\star(x)}(\mathbb{R}^N)$. As a consequence, we can no longer expect the Palais–Smale condition to hold uniformly. However,  
	 we can prove a local Palais–Smale condition that will hold for $E_{\uplambda}(u,v)$ below a certain value of energy, by using the principle of concentration compactness for the weighted-variable exponent Sobolev space $W^{1,p(x)}_{b}(\mathbb{R}^N)$. For  reader's convenience, we state this result in order to prove Theorem \ref{theooo}, see Fu and Zhang \cite[ Theorem 2.2]{Fu1} for the proof.\par 
	 Let us recall that $\mathcal{M}_{B}(\mathbb{R}^N)$ denotes the space of finite nonnegative Borel measures on $\mathbb{R}^N$. For any  $\upnu \in \mathcal{M}_{B}(\mathbb{R}^N)$, we have$\upnu (\mathbb{R}^N)=\left\| \upnu\right\| $. We say that $\upnu_n \overset{\ast}{\rightharpoonup} \upnu$ weakly-* in $\mathcal{M}_{B}(\mathbb{R}^N)$ if, as $n\to \infty$, we have 	$(\upnu_n,\upxi)\to (\upnu,\upxi)$ for all $\upxi \in C_0(\mathbb{R}^N)$. Therefore, just as in Fu and Zhang \cite[ Theorem 2.2]{Fu1}, we can readily deduce the following.
	\begin{theorem}[]\label{ccpp}
		Consider  $h\in  C_{+}^{\text{log}}(\mathbb{R}^N) $ and $q\in C(\mathbb{R}^N)$  such that
		$$
		1<\inf_{x\in\mathbb{R}^N}p(x)\le \sup_{x\in\mathbb{R}^N}h(x) < N \quad
		\text{and}\quad 1\le r(x)\le h^*(x)\quad \text{for all $x$ in }	\mathbb{R}^N.
		$$
		Let $\{u_n\}_{n\in\mathbb{N}}$ be a sequence that weakly converges in
		$W^{1,h(x)}_{b}(\mathbb{R}^N)$ to $u$, and such that $\|u_n\|_{b,h}\leq 1$. The sequence satisfies the following conditions as $n\to \infty$
		
		\begin{itemize}
			\item $|\nabla u_n|^{h(x)}+ b(x)|u_n|^{h(x)}\overset{\ast}{\rightharpoonup} \upmu$
			in $\mathcal{M}_{B}(\mathbb{R}^N)$.
			
			\item $|u_n|^{r(x)}\overset{\ast}{\rightharpoonup} \upnu$  
			in $\mathcal{M}_{B}(\mathbb{R}^N)$.
		\end{itemize}
		as $n\to \infty$.  	Additionally, suppose that $\mathcal{C}_{h}= \{x\in 	\mathbb{R}^N\colon r(x)=h^*(x)\}$ is
		nonempty. Then for some countable index set $I$, we have
		
		\begin{gather} 
				\upmu = |\nabla u|^{h(x)} + b(x) | u|^{h(x)}+\sum_{i\in I} \upmu_i \delta_{x_i} +\tilde{\upmu},   \quad \upmu (\mathcal{C}_{h})\leq 1;\label{cp12}\\
			\upnu=|u|^{r(x)} + \sum_{i\in I}\upnu_i\delta_{x_i},\quad \upnu(\mathcal{C}_{h})\leq S;\label{cp11}
		\end{gather}
		with $$S=\sup\left\lbrace \int_{\mathbb{R}^N}\left| u\right|^{r(x)}dx: u\in W^{1,h(x)}_{b}, \left\| u\right\|_{b,h} \leq 1\right\rbrace,
		\
		\{x_i\}_{i\in I}\subset \mathcal{C}_{h},
		$$  and  $\left\lbrace \upmu_i\right\rbrace ,\left\lbrace \upnu_i\right\rbrace\subset [ 0,\infty)$,  $\delta_{x_i}$ is the Dirac mass at  $x_i \in \mathcal{C}_{h}$,  and $\tilde{\upmu} \in \mathcal{M}_{B}(\mathbb{R}^N)$ is a
		non-atomic non-negative measure. The atoms and the regular part satisfy the
		generalized Sobolev inequality 
		$$ \upnu(\mathcal{C}_{h})\leq 2^{(h^+r^+)/h^-} C^{\ast} \max\left\lbrace \upmu(\mathcal{C}_{h})^{r^+/h^-},\upmu(\mathcal{C}_{h})^{r^-/h^+}\right\rbrace 
		$$
		and
		$$\upnu_i\leq C^{\ast}\max \left\lbrace \upmu_i^{r^+/h^-}, \upmu_i^{r^-/h^+} \right\rbrace .$$
	
	\end{theorem}
	
   Theorem \ref{ccpp} does not account for the potential mass loss at infinity within a weakly convergent sequence. Subsequently, Theorem \ref{ccppp} quantifies this occurrence. In a manner reminiscent of Fu and Zhang \cite[Theorem 2.5]{Fu1}, we deduce the following outcomes.
	\begin{theorem}[]\label{ccppp}
		Assuming  $\mathcal{C}_{h}= \{x\in 	\mathbb{R}^N\colon r(x)=h^*(x)\}$, let  $\{u_n\}_{n\in\mathbb{N}}$ be a sequence that weakly converges in $W^{1,h(x)}_{b}(\mathbb{R}^N)$ to $u$, and such that :
		\begin{itemize}
			\item $|\nabla u_n|^{h(x)}+ b(x)|u_n|^{h(x)}\overset{\ast}{\rightharpoonup}\upmu$
			in $\mathcal{M}_{B}(\mathbb{R}^N)$ .
			
			\item $|u_n|^{r(x)}\overset{\ast}{\rightharpoonup}\upnu$ 
			in $\mathcal{M}_{B}(\mathbb{R}^N)$ .
		\end{itemize}
	We define the quantities: 
		$$ \upmu_{\infty}=\lim_{\left\lbrace x\in \mathbb{R}^N ; 
			|x|>R\right\rbrace }\limsup_{n\to +\infty }\int_{\left\lbrace x\in \mathbb{R}^N ; |x|>R\right\rbrace} ( |\nabla u_n|^{h(x)}+ b(x)|u_n|^{h(x)})dx$$
		and 
		$$ \upnu_{\infty}=\lim_{R\to \infty}\limsup_{n\to +\infty }\int_{\left\lbrace x\in \mathbb{R}^N ; 
			|x|>R\right\rbrace } |u_n|^{r(x)}dx.$$
		The quantities $\upmu_{\infty}$ and $\upnu_{\infty}$ are well defined and satisfy		
		$$\limsup_{n\to +\infty }\int_{\mathcal{C}_{h}} ( |\nabla u_n|^{h(x)}+ b(x)|u_n|^{h(x)})dx= \int_{\mathcal{C}_{h}}d\upmu +\upmu_{\infty}
		\
\hbox{and}
\ 
\limsup_{n\to +\infty }\int_{\mathcal{C}_{h}}  |u_n|^{r(x)}dx= \int_{\mathcal{C}_{h}}d\upnu +\upnu_{\infty}.$$
		
		Additionally,  the following inequality holds
		$$\upnu_\infty\leq C^{\ast}\max \left\lbrace \upmu_\infty^{r^+/h^-}, \upmu_\infty^{r^-/h^+} \right\rbrace .$$
	\end{theorem}	
	
\subsection{Limit index theory}\label{subsect22}

In this subsection, we shall introduce the Limit index theory due to Li \cite{Li1}. In order to do that, we recall the following definitions (the interested readers can refer to  Szulkin \cite{SZULIUN} and Willem \cite{Willem}).
\begin{definition}[see \cite{Li1}]
	
	The action of a topological group $G$  on a normed space $Z$ is a continuous map $ G\times Z\to Z:\left[ g,z\right] \mapsto gz$
	such that  
	$$1.z=z, \quad (gh)z=g(hz), \qquad z\mapsto gz \text{ is linear }  \  \hbox{for all} \  g,h \in G.$$
	The action is isometric if $\left\|gz \right\|=\left\|z \right\|, \ \  \hbox{for all} \  g\in G, z\in Z $ in which case, $Z$ is called the $G$-space.
	The set of invariant points is defined by
	$$\text{Fix}G:= \left\lbrace  z \in Z : gz=z \text{ for all } g\in G\right\rbrace. $$
	A set $A  \subset Z$  is invariant if $gA=A$ for every $g\in G$. A function $\varphi :Z\to R$ is invariant if $\varphi\circ g=\varphi$ for every $g\in G, z\in Z.$ 	
	A map $f:Z\to Z$  is equivariant if $g\circ f= f\circ g \text{ for every } g\in G.$
	
		\noindent
	Suppose that $Z$ is a $G$-Banach space, that is, there is a $G$ isometric action on $Z$. Let 	
	$$\Sigma:= \left\lbrace A\subset Z: A \text{ is closed and } gA=A, \text{ for all } g\in G \right\rbrace $$
	be a family of all $G$-invariant closed subset of $Z$, and let 
	$$h:= \left\lbrace  h \in C^{0}(Z,Z) : h(gu)=g(hu), \text{ for all } g \in G\right\rbrace$$
	be the class of all $G$-equivariant mappings of $Z$. Finally, the set
	$O(u):=\left\lbrace gu:g\in G\right\rbrace $
	is called the
	$G$-Orbit of $u$.
\end{definition}

\begin{definition}[see \cite{Li1}]\label{diif}
	An index for $(G,\Sigma,h)$ is a mapping $i: \Sigma \to \mathbb{Z}_+\cup\left\lbrace +\infty\right\rbrace ,$
	where $\mathbb{Z}_+$is the set of all nonnegative integers,
	 such that for all $A, B\in \Sigma, h\in h$,the following conditions are satisfied
	
	\begin{enumerate}
		\item $i(A)=0 \Leftrightarrow A=\emptyset$;
		\item (Monotonicity) $A \subset B\Rightarrow i(A)\leq i(B)$;
		\item (Subadditivity) $i(A\cup B)\leq i(A)+i(B)$;
		\item (Supervariance) $i(A) \leq i(\overline{h(A)}), \  \hbox{for all} \  h \in h$;
		\item (Continuity) If $A$ is compact and $A\cap \text{Fix}G=\emptyset$, then $i(A)<+\infty$ and there is a $G$-invariant neighborhood $N$ of $A$ such that $i(\overline{N})=i(A)$;
		\item (Normalization) If $x\notin \text{Fix}G$, then $i(O(x))=1$.
	\end{enumerate}
\end{definition}

\begin{definition}[see \cite{Benci}]\label{difff}
	An index theory is said to satisfy the $d$-dimension property if there is a positive integer $d$ such that
	$i(V^{dk}\cap S_{1}(0))=k,$
	for all $dk$-dimensional subspaces $V^{dk}\in \Sigma$ such  that $V^{dk}\cap\text{Fix}G=\left\lbrace 0\right\rbrace$, where $S_{1}(0))$ is the unit sphere in $Z$.
	
	\noindent
	Suppose $U$ and $V$ are $G$-invariant closed subspaces of $Z$ such that $Z=U\oplus V$, where $V$ is infinite dimensional and
	$$ V=\overline{\bigcup_{j=1}^{\infty}V_j},$$
	where $V_j$ is $dn_j$-dimensional $G$-invariant subspaces of $V$,
	$j=1,2,...$, and $V_1\subset V_2 \subset \dots \subset V_n \subset\dots $. Let 
	$Z_j= U\oplus V_j$
	and for all $A\in \Sigma$, let $A_j = A\cap Z_j.$
\end{definition}
\begin{definition} [see \cite{Li1}]\label{diiif}
	Let $i$ be an index theory satisfying the $d$-dimension property. A limit index with respect to $(Z_j)$ induced by $i$ is a mapping
	$i^{\infty}: \Sigma\to \mathbb{Z}\cup \left\lbrace -\infty;+\infty \right\rbrace ,$
	given by $i^{\infty}(A)=\limsup_{j\to \infty}(i(A_j)-n_j).$	
\end{definition}
\begin{proposition}[see \cite{Li1}]
	Let $A, B\in \Sigma$. Then $i^{\infty}$ satisfies the following:
	\begin{enumerate}
		\item $A=\emptyset \Rightarrow i^{\infty}=-\infty$;
		\item (Monotonicity) $A \subset B\Rightarrow i^{\infty}(A)\leq i^{\infty}(B)$;
		\item (Subadditivity) $i^{\infty}(A\cup B)\leq i^{\infty}(A)+i^{\infty}(B)$;
		
		\item  If $ V\cap \text{Fix}G=\left\lbrace 0\right\rbrace$, then 
		$i^{\infty}(S_{\rho}(0)\cap V)=0$, where $S_{\rho}(0)=\left\lbrace z\in Z: \left\| z\right\| =\rho \right\rbrace $;
		
		\item If $Y_0$ and $\tilde{Y}_0$ are $G$-invariant closed subspaces of $V$ such that $V=Y_0\oplus\tilde{Y}_0$, $\tilde{Y}_0 \subset V_{j_0}$ for some $j_0$ and $\dim\tilde{Y}_0=dm$, then $i^{\infty}(S_{\rho}(0)\cap Y_0)\geq -m$.
	\end{enumerate}
\end{proposition}

\begin{definition}[see \cite{Willem}]	
	A functional $E\in C^{1}(Z,\mathbb{R})$ is said to satisfy 
	 condition $(PS)_c$ if any sequence $\{u_{nk}\}_{k}$, $u_{nk}\in Z_{nk}$, such that $ E_{n_k}(u_{nk}) \to c, \quad E'_{nk}(u_{nk})\to 0, \text{ as }
	n_k\to \infty,$	
	possesses a convergent subsequence, where $Z_{nk}$ is the $n_k$-dimensional subspace of $Z$ as in Definition \ref{difff} and
	$E_{n_k}=E|Z_{nk}$.	
\end{definition}
\begin{theorem}[see \cite{Li1}]\label{thhh}
	Assume the following:
	\begin{itemize}
		\item[$(\textbf{\textit{B}}_1)$]  $E\in C^1(Z,\mathbb{R})$ is $G$-invariant;		
		\item[$(\textbf{\textit{B}}_2)$] There exist G-invariant closed subspaces $U$ and $V$ such that $V$ is infinite dimensional and
		$Z=U\oplus V$;
		\item[$(\textbf{\textit{B}}_3)$] There is a sequence of $G$-invariant finite-dimensional subspaces 
		$$  V_1\subset V_2 \subset \dots \subset V_j \subset \dots , \quad \dim V_j =dn_j,$$
		such that $V=\overline{\bigcup_{j=1}^{\infty}V_j}$;
		\item[$(\textbf{\textit{B}}_4)$] There is an index theory $i$ on $Z$ satisfying the $d$-dimension property;
		\item[$(\textbf{\textit{B}}_5)$]
		There are $G$-invariant subspaces $Y_0, \tilde{Y}_0, Y_1$ of $V$ such that $V=Y_0\oplus \tilde{Y}_0$, $Y_1, \tilde{Y}_0\subset V_{j_0}$ for some $j_0$ and $\dim \tilde{Y}_0= dm <dk=\dim Y_1$;		
		\item[$(\textbf{\textit{B}}_6)$] There are $\mathfrak{M}$ and $\mathfrak{N} $, $\mathfrak{M} <\mathfrak{N}$ such that $E$ satisfies $(PS)_c$, for all $c\in \left[ \mathfrak{M},\mathfrak{N}\right] $;
		\item[$(\textbf{\textit{B}}_7)$] 		 The following holds
		\begin{eqnarray*}
			\begin{cases}
				(1) \text{ either } \text{Fix}G\subset U\oplus Y_1 \text{ or }   \text{Fix}G\cap V=\left\lbrace 0\right\rbrace ,\\
				(2) \text{ there is } \rho>0 \text{ Such that for all } u\in Y_0\cap S_{\rho}(0), \text{ we have } E(z)\geq \mathfrak{M},\\
				(3) \text{ for all } z\in U\oplus Y_1, \text{ we have } E(z)\leq \mathfrak{N}.
			\end{cases}
		\end{eqnarray*}		
		If $i^{\infty}$  is the limit index corresponding to $i$, then the numbers
		$ c_j :=\inf_{i^{\infty}(A)\geq j}\sup_{z\in A}E(z), \ -k+1\leq j\leq -m,$
	\end{itemize}		
	are critical values of $E$ , and $\mathfrak{M} \leq c_{-k+1}\leq \dots \leq c_{-m} \leq \mathfrak{N}$.
	Moreover, if $c= c_l= \dots=c_{l+r}$, $r\geq 0$, then $i(\mathbb{K}_c)\geq r+1$, where $\mathbb{K}_c=\left\lbrace z\in Z: E'(z)=0, E(z)=c\right\rbrace .$
\end{theorem}
\textbf{\textbf{Notations.}}  $X=W^{1,p(x)}_{b}(\mathbb{R}^N)\cap W^{1,h(x)}_{b}(\mathbb{R}^N),$
 $Y=X\times X,$
 $G_1= O(\mathbb{N})$ is the group of orthogonal linear transformations in $\mathbb{R}^N,$
$$X_{G_1}:=W^{1,p(x)}_{b,G_1}(\mathbb{R}^N)\cap W^{1,h(x)}_{b,G_1}(\mathbb{R}^N)=\{u\in W^{1,p(x)}_{b}(\mathbb{R}^N)\cap W^{1,h(x)}_{b}(\mathbb{R}^N): gu(x)=u(g^{-1}x)=u(x), g\in G_1\}
$$
$$
= \{u\in W^{1,p(x)}_{b}(\mathbb{R}^N)\cap W^{1,h(x)}_{b}(\mathbb{R}^N): u \text{ and } v \text{ are radially symmetric}\} ,$$ and $Z=Y_{G_1}= X_{G_1}\times X_{G_1}$.

\section{Proof of the main result}\label{sect3}

To prove the main result of this paper,  Theorem \ref{theooo}, we 
shall perform
 a careful analysis of the behavior of minimizing sequences in Lemma \ref{lemma2}, by using  the concentration-compactness principle for the weighted-variable exponent Sobolev space stated above, which will
 allow us to recover compactness below some critical threshold.

Since $\tilde{E}_{\uplambda}\in C^1(Y,\mathbb{R})$,  the weak solutions for problem (\ref{s1.1}) coincide with the critical
points of $\tilde{E}_{\uplambda}$. On the other hand, by condition$(\textbf{\textit{F}}_5)$, it is immediate that $\tilde{E}_{\uplambda}$ is $G_1$-invariant. Therefore by the
principle of symmetric criticality of Krawcewicz and Marzantowicz \cite{Krawcewicz}, we know that
$(u,v)$ is a critical point of  $\tilde{E}_{\uplambda}$
if and only if $(u,v)$ is a critical point of
$E_{\uplambda}=\tilde{E}_{\uplambda}|_{Z=X_{G_1}\times X_{G_1}}$. 
So it suffices to prove the
existence of a sequence of critical points of $E_{\uplambda}$ on $Z$.\par 

Let $X$ be a Banach space and a functional $f\in C^1(X,\mathbb{R})$. Given sequence $\left\lbrace u_n\right\rbrace _n$ in $X$, if there exist $c\in \mathbb{R}$ such that 
\begin{equation*}
	f(u_n) \to c   \text{ and } f'(u_n)\to 0 \text{ in } X',
\end{equation*}
we say that $\left\lbrace u_n\right\rbrace _n$ is a Palais-Smale sequence with energy level $c$ ( or $\left\lbrace u_n\right\rbrace$ is $(PS)_c$ for short).
When any $(PS)_c$ sequence for $f$ possesses some strongly convergent subsequence in $X$, we say that $f$ satisfies the Palais-Smale condition at level $c$ (or $f$ is $(PS)_c$ short).\par 

In order to prove that $E_{\uplambda}$ satisfies  $(PS)_c$ we recall some properties of the Banach space $X$. According to Triebel \cite[Sect. 4.9.4]{TRIEBEL}, there exists a Schauder basis $\left\lbrace e_n'\right\rbrace_{n=1}^{\infty}$ for $X$.
Let $e_n=\int_{G_1}e_n'(g(x))d_{\upmu_g}.$
We are going to, 
 if necessary, select one in identical elements. 
 We know that $\left\lbrace e_n\right\rbrace_{n=1}^{\infty}$
is a Schauder basis for $X_{G_1}$. Furthermore, since $X_{G_1}$ is reflexive, 
$\left\lbrace e_n^*\right\rbrace_{n=1}^{\infty}$ the biorthogonal functionals associated to the basis $\left\lbrace e_n\right\rbrace_{n=1}^{\infty}$
(which is characterized by relations 
$$ \left\langle e_m^*,e_n\right\rangle =\delta_{m,n}=\begin{cases}
	1  \text{ if } n=m ,\\
	~0 \text{ if } n\neq m.
\end{cases}),$$
 form a basis for $X_{G_1}^\star$ with the following properties - see Lindenstrauss and Tzafriri \cite[ cf. Proposition l.b.1 and Theorem l.b.51]{LINDENSTRAUSS}. Denote
$$X_{G_1}^{(n)}= \text{span}\left\lbrace e_1,...,e_n\right\rbrace ,\
X_{G_1}^{(n)^\perp}= \overline{\text{span}\left\lbrace e_{n+1},...\right\rbrace}, 
\
X_{G_1}^{*(n)}= \text{span}\left\lbrace e_1^\star,...,e_n^\star\right\rbrace.$$

Let $P_n: X_{G_1}\to X_{G_1}^{(n)}$  be the projector corresponding to decomposition $X_{G_1}= X_{G_1}^{(n)}\oplus X_{G_1}^{(n)^\perp}$ and $P_n^{*}: X_{G_1}^*\to X_{G_1}^{*(n)}$
the projector corresponding to the decomposition $X_{G_1}^*= X_{G_1}^{*(n)}\oplus X_{G_1}^{*(n)^\perp}$. Then $P_nu\to u$, $P_n^*v^*\to v^*$ for any $u\in X_{G_1}$, $v^*\in X_{G_1}^*$ as $n\to \infty$ and 
$\left\langle P_n^*v^*,u\right\rangle =\left\langle v^*,P_nu\right\rangle$.
Set 
$ Z=Y_{G_1}=X_{G_1}\times X_{G_1}, \quad Z_{n}=X_{G_1}\times X_{G_1}^{(n)}.$
We shall prove the following local Palais-Smale condition.
\begin{lemma} \label{lemma2}
	Assume that the conditions
	 $(\textbf{\textit{H}}_{a_1})-(\textbf{\textit{H}}_{a_4})$, $(\textbf{\textit{K}}_1)-(\textbf{\textit{K}}_2)$  and  $(\textbf{\textit{F}}_1)-(\textbf{\textit{F}}_5)$ 
	 are satisfied. Then the functional $E_{\uplambda}$ satisfies the local $(PS)_c$ with 
	$$c \in \Biggl(-\infty, \Big( \dfrac{1}{\vartheta^-}-\frac{1}{r^-}\Big)\max\Big\{\Big(\frac{\mathfrak{K}_0  D}{ S^{\frac{h^-}{r^+}}}\Big)^{\frac{r^+}{r^+-h^-}},\Big(\frac{\mathfrak{K}_0D}{ S^{\frac{h^-}{r^-}}}\Big)^{\frac{r^-}{r^--h^-}}\Big\}\Biggr),$$
	where $D=(1-\mathcal{H}(\kappa_\star^3)) \min \left\{\kappa_1^0,\kappa_2^0\right\}+ \mathcal{H}(\kappa_\star^3)\min \left\{\kappa_1^2,\kappa_2^2\right\}$, in the following sense: if $\left\lbrace y_{n_k}\right\rbrace \subset Y$ is a sequence such that $y_{n_k}=(u_{n_k},v_{n_k})$ and  
	$$E_{\uplambda_{ n_k}}(u_{n_k},v_{n_k})\to c 
	\text{ and }  E'_{\uplambda_{n_k}}(u_{n_k},v_{n_k}) \to 0, \text{ as }  k\to \infty,$$ 
	where $	E_{\uplambda_{n_k}}=E_{\uplambda}|_{Z_{n_k}}$ with $Z_{n_k}= X_{G_1}\times X_{n_k},$
	then $\left\lbrace (u_{n_k},v_{n_k})\right\rbrace_k $  possesses a subsequence converging strongly  in $Z$ to a critical point of the functional $E_{\uplambda}$.
	
\end{lemma}

\begin{proof}
	
	First, we show that  $\left\lbrace (u_{n_k},v_{n_k})\right\rbrace $ is bounded in $Z$.
	If not, we may assume that $ \| u_{n_k}\|_{b,h}>1$ and $ \| v_{n_k}\|_{b,h }>1$ for any
	integer $n$. We have by condition $(\textbf{\textit{F}}_3)$,
	
	\begin{align*}
		o(1)\| u_{n_k}\|_{b,h}&\geq \langle -E'_{\uplambda_{n_k}}(u_{n_k},v_{n_k}),(u_{n_k},0) \rangle	\\
		&=  K\left(
		\mathcal{B}(u_{n_k})\right)\int_{\mathbb{R}^N }  \Big(\mathcal{A}_1(\nabla u_{n_k}) .\nabla u_{n_k} 
		+  b(x)\mathcal{A}_2(u)u \Big)dx\\
		&\quad+  \int_{\mathbb{R}^N} |u_{n_k}|^{r(x)}dx+\int_{\mathbb{R}^N}\uplambda(x)\frac{ \partial \mathcal{F}}{\partial u}(x,u_{n_k},v_{n_k})u_{n_k}dx,
		\\
		&\geq         K\left(
		\mathcal{B}(u_{n_k})\right)\int_{\mathbb{R}^N}\Big(a_1(|\nabla u_{n_k}|^{p(x)})|\nabla u_{n_k}|^{p(x)}+ b(x)a_2(|u_{n_k}|^{p(x)})| u_{n_k}|^{p(x)}\Big)dx.    
	\end{align*}	
	Therefore, by using $(\textbf{\textit{K}}_1)$ and $(\textbf{\textit{H}}_{a_2})$, we have 
	\begin{multline}\label{key}
		\langle -E'_{\uplambda_{n_k}}(u_{n_k},v_{n_k}),(u_{n_k},0) \rangle	\geq  \mathfrak{K}_0 \Big[ \min \left\{\kappa_1^0,\kappa_2^0\right\}\int_{\mathbb{R}^N} \Big(|\nabla u_{n_k}|^{p(x)}+ b(x)| u_{n_k}|^{p(x)}\Big)dx\\
		\quad +\min \left\{\kappa_1^2,\kappa_2^2\right\}\mathcal{H}(\kappa_\star^3)\int_{\mathbb{R}^N} \Big(|\nabla u_{n_k}|^{q(x)}+ b(x)| u_{n_k}|^{q(x)}\Big)dx\Big].
	\end{multline}	
	Let us assume, for the sake of contradiction, that there exists a subsequence, still denoted by $\left\lbrace u_{n_k}\right\rbrace $, such that 
	$\left\|  u_{n_k}\right\|_{b,h}
	\to + \infty $.	
	If $\kappa_\star^3=0$, from Proposition \ref{prop2}, we have 
	\begin{align*}
		\langle -E'_{\uplambda_{n_k}}(u_{n_k},v_{n_k}),(u_{n_k},0) \rangle&\geq  C_{1}  \| u_{n_k}\|_{b,p}^{p^{-}},
	\end{align*}	
	thus 
	\begin{align*}
		o(1)\| u_{n_k}\|_{b,h}&\geq C_{1}  \| u_{n_k}\|_{b,p}^{p^{-}}.
	\end{align*}	
   However, this is a contradiction since $p^->1$. Therefore we can
	conclude that $\left\lbrace u_{n_k}\right\rbrace $ is bounded in $X_{G_1}$.	
	
	Next, if $\kappa_\star^3>0$, we need to analyze the following cases :\\	
	
	$(1)  \left\|  u_{n_k} \right\| _{b,p} \to +\infty$ and $\left\|  u_{n_k}\right\| _{b,q} \to +\infty$  as $k \to  +\infty$;
	
	$(2)  \left\|  u_{n_k} \right\| _{b,p}  \to +\infty$  and $\left\|  u_{n_k}\right\| _{b,q} $ is bounded;

	$(3)  \left\|  u_{n_k} \right\| _{b,p} $ is bounded and $\left\|  u_{n_k}\right\| _{b,q} \to +\infty$.
	
	\noindent
	We shall investigate each of these cases separately. In the case $(1)$, for $m$ large enough, $\| u_{n_k}\|_{b,q} ^{q^-} \geq \| u_{n_k}\|_{b,q} ^{p^-}$.   Hence, using relation \eqref{key}, we get 
	\begin{align*}
		c+ o_{n_k}(1)&\geq  C_{1} \| u_{n_k}\|_{b,p}^{p^{-}} +C_{2}\mathcal{H}(\kappa_\star^3)\| u_{n_k}\|_{b,q}^{q^{-}}\\
		&\geq  C_{1} \| u_{n_k}\|_{b,p}^{p^{-}} +C_{2}\mathcal{H}(\kappa_\star^3)\| u_{n_k}\|_{b,q}^{p^{-}},
		\geq C_{3} \| u_{nk}\|_{b,p}^{p^{-}},
	\end{align*}
	which leads to an absurd result.
	
	\noindent
	In the case (2), by relation \eqref{key}, we obtain
	\begin{align*}
		c+ o_{n_k}(1)&\geq C_{1} \| u_{n_k}\|_{b,p}^{p^{-}}.
	\end{align*}
	Taking the limit as $k\to +\infty$, we get a contradiction.
	
	\noindent
	Case $(3)$ can be handled in a manner similar to case $(2)$. 
	Hence, we conclude that $\left\lbrace u_{n_k}\right\rbrace $ is bounded in $X_{G_1}$.
	
	On the one hand, we get
	\begin{align*}
		c+o(1)\| v_{n_k}\|_{b,h}&\geq E_{\uplambda_{n_k}}(0,v_{n_k})-\langle E'_{\uplambda_{n_k}}(u_{n_k},v_{n_k}),(0,\frac{v_{n_k}}{\vartheta(x)}) \rangle	\\
		&= \widehat{K}\left(
		\mathcal{B}(v_{n_k})\right) -\int_{\mathbb{R}^N}\frac{1}{r(x)}| v_{n_k}|^{r(x)}dx-\int_{\mathbb{R}^N}\uplambda(x)\mathcal{F}(x,0,v_{n_k})dx\\
		&\quad -K\left(\mathcal{B}(v_{n_k})\right)\int_{\mathbb{R}^N }  \Big(\mathcal{A}_1(\nabla v_{n_k})\nabla(\frac{v_{n_k}}{\vartheta(x)})dx + b(x)\mathcal{A}_2(v_{n_k}) \frac{v_{n_k}}{\vartheta(x)}\Big)dx  \\
		&\quad+\int_{\mathbb{R}^N}\frac{1}{\vartheta(x)}| v_{n_k}|^{r(x)}dx +\int_{\mathbb{R}^N}\uplambda(x)\frac{ \partial \mathcal{F}}{\partial v}(x,0,v_{n_k})\frac{v_{n_k}}{\vartheta(x)}dx,
	\end{align*}
	that is, 	
	\begin{align*}
		c+o(1)\| v_{n_k}\|_{b,h}	&= \widehat{K}\left(
		\mathcal{B}(v_{n_k})\right) -K\left(\mathcal{B}(v_{n_k})\right)\int_{\mathbb{R}^N }  \Big(\mathcal{A}_1(\nabla v_{n_k})\nabla(\frac{v_{n_k}}{\vartheta(x)}) 
	 +b(x)\mathcal{A}_2(v_{n_k}) \frac{v_{n_k}}{\vartheta(x)}\Big)dx  \\	&\quad+\int_{\mathbb{R}^N}\Big( \dfrac{1}{\vartheta(x)}-\frac{1}{r(x)}\Big)| v_{n_k}|^{r(x)}dx +  \int_{\mathbb{R}^N}\uplambda(x)\Big(\frac{ \partial \mathcal{F}}{\partial v}(x,0,v_{nk})\frac{v_{n_k}}{\vartheta(x)}
		-\mathcal{F}(x,0,v_{n_k})\Big)dx.
	\end{align*}	
	
	\noindent
	Next, by using  $(\textbf{\textit{H}}_{a_4}), (\textbf{\textit{K}}_1)-(\textbf{\textit{K}}_2)$ and  $(\textbf{\textit{F}}_2)$, we obtain	
	\begin{align*}
		c+o(1)\| v_{n_k}\|_{b,h}	&\geq\mathfrak{K}_0\int_{\mathbb{R}^N}\Bigg(\Big( \dfrac{\gamma }{\alpha_1p(x)}-\frac{1}{\vartheta(x)}\Big)
		a_1(|\nabla v_{n_k}|^{p(x)})|\nabla v_{n_k}|^{p(x)}\\ &\quad+\Big( \dfrac{\gamma }{\alpha_2p(x)}-\frac{1}{\vartheta(x)}\Big)b(x)a_2(| v_{n_k}|^{p(x)})| v_{n_k}|^{p(x)}\Bigg)dx\\
		&\quad +\mathfrak{K}_0\int_{\mathbb{R}^N}\dfrac{v_{n_k}}{\vartheta(x)^2}a_1(|\nabla v_{n_k}|^{p(x)})|\nabla v_{n_k}|^{p(x)-2}\nabla v_{n_k}\nabla\vartheta dx +\int_{\mathbb{R}^N}\Big( \dfrac{1}{\vartheta(x)}-\frac{1}{r(x)}\Big)| v_{n_k}|^{r(x)}dx,\\
		&\geq \mathfrak{K}_0 \int_{\mathbb{R}^N}\Big(\frac{\gamma}{\max\{\alpha_1,\alpha_2\}p(x)}-\frac{1}{\vartheta(x)}\Big)\Big[a_1(|\nabla v_{n_k}|^{p(x)})|\nabla v_{n_k}|^{p(x)}+b(x)a_2(| v_{n_k}|^{p(x)})| v_{n_k}|^{p(x)}\Big]dx\\
		&\quad +\mathfrak{K}_0\int_{\mathbb{R}^N}\dfrac{v_{n_k}}{\vartheta(x)^2}a_1(|\nabla v_{n_k}|^{p(x)})|\nabla v_{n_k}|^{p(x)-2}\nabla v_{n_k}\nabla\vartheta dx +\int_{\mathbb{R}^N}\Big( \dfrac{1}{\vartheta(x)}-\frac{1}{r(x)}\Big)| v_{n_k}|^{r(x)}dx.\\
	\end{align*}
	
	\noindent
	Denote
	\begin{equation*}
		\eth_1:=\inf_{x\in \mathbb{R}^N}\Big( \dfrac{\gamma}{\max\{\alpha_1,\alpha_2\}p(x)}-\frac{1}{\vartheta(x)}\Big)>0, \text{ and }
		\eth_2:=\inf_{x\in \mathbb{R}^N}\Big( \dfrac{1}{\vartheta(x)}-\frac{1}{r(x)}\Big)>0.
	\end{equation*}
	Then, using $(\textbf{\textit{F}}_2)$, we get
	\begin{align*}
		E_{\uplambda_{n_k}}(0,v_{k_n})-\langle E'_{\uplambda_{n_k}}(u_{n_k},v_{n_k}),(0,\frac{v_{n_k}}{\vartheta(x)}) \rangle &\geq   \mathfrak{K}_0\Bigg( \eth_1\min \left\{\kappa_1^0,\kappa_2^0\right\}\int_{\mathbb{R}^N} \Big(|\nabla v_{n_k}|^{p(x)}+ b(x)| v_{n_k}|^{p(x)}\Big)dx\\
		&\quad +\eth_1\min \left\{\kappa_1^2,\kappa_2^2\right\}\mathcal{H}(\kappa_\star^3)\int_{\mathbb{R}^N} \Big(|\nabla v_{n_k}|^{q(x)}+ b(x)| v_{n_k}|^{q(x)}\Big)dx\Bigg)	\\
		& +\mathfrak{K}_0\int_{\mathbb{R}^N}\dfrac{v_{n_k}}{\vartheta(x)^2}a_1(|\nabla v_{n_k}|^{p(x)})|\nabla v_{n_k}|^{p(x)-2}\nabla v_{n_k}\nabla\vartheta dx
		+\int_{\mathbb{R}^N}\eth_2| v_{n_k}|^{r(x)}dx.
	\end{align*}
	On the other hand, we obtain
	\begin{align*}
		\Bigg|\dfrac{v_{n_k}}{\vartheta(x)^2}a_1(|\nabla v_{n_k}|^{p(x)})|\nabla v_{n_k}|^{p(x)-2}\nabla v_{n_k}\nabla\vartheta \Bigg|&\leq \Bigg|\kappa_1^1\dfrac{v_{n_k}}{\vartheta(x)^2}|\nabla v_{n_k}|^{p(x)-2}\nabla v_{n_k}\nabla\vartheta \Bigg|+\Bigg|\kappa_\star^3\dfrac{v_{n_k}}{\vartheta(x)^2}|\nabla v_{n_k}|^{q(x)-2}\nabla v_{n_k}\nabla\vartheta \Bigg|.
	\end{align*}
	
	\noindent
	By use the Young inequality, for any $\varepsilon\in (0, 1)$, there exists $c_1(\varepsilon)$ and $c_2(\varepsilon) > 0$ such that
	\begin{equation}\label{yoni1}
		\Bigg|\dfrac{v_{n_k}}{\vartheta(x)^2}|\nabla v_{n_k}|^{p(x)-2}\nabla v_{n_k}\nabla\vartheta \Bigg| \leq \varepsilon|\nabla v_{n_k}|^{p(x)}+ c_1(\varepsilon)|v_{n_k}|^{p(x)},
	\end{equation}	
	
	\begin{equation}\label{yoni2}
		\Bigg|\dfrac{v_{n_k}}{\vartheta(x)^2}|\nabla v_{n_k}|^{q(x)-2}\nabla v_{n_k}\nabla\vartheta \Bigg| \leq \epsilon|\nabla v_{n_k}|^{q(x)}+ c_2(\epsilon)| v_{n_k}|^{q(x)}.
	\end{equation}
	Hence, by relations (\ref{yoni1}) and  (\ref{yoni2}), we get
	\begin{align*}
		c+o(1)\| v_{n_k}\|_{b,h}&\geq  \mathfrak{K}_0 \Bigg(\int_{\mathbb{R}^N} \Big((\eth_{\star}-\epsilon)|\nabla v_{n_k}|^{p(x)}+ (\eth_{\star}b(x)-c_1(\epsilon))| v_{n_k}|^{p(x)}\Big)dx\\
		&\quad \mathcal{H}(\kappa_\star^3)\int_{\mathbb{R}^N} \Big((\eth_{\star}-\epsilon)|\nabla v_{n_k}|^{q(x)}+ (\eth_{\star}b(x)-c_2(\epsilon))| v_{n_k}|^{q(x)}\Big)dx\Bigg),
	\end{align*}
	where $ \eth_{\star}= \min \left\lbrace  \eth_1\min \left\{\kappa_1^0,\kappa_2^0\right\},  \eth_1\min \left\{\kappa_1^2,\kappa_2^2\right\}\right\rbrace $.
	
	\noindent
	Let $\epsilon< \eth_{\star}/2$ and $w_0=2\max(c_1(\epsilon),c_2(\epsilon))/ \eth_{\star}$,  we get
	\begin{align*}
		c+o(1)\| v_{n_k}\|_{b,h}&\geq  \mathfrak{K}_0 \frac{ \eth_{\star}}{2}\Bigg(\int_{\mathbb{R}^N} \Big(|\nabla v_{n_k}|^{p(x)}+b(x)| v_{n_k}|^{p(x)}\Big)dx\\
		&\quad \mathcal{H}(\kappa_\star^3)\int_{\mathbb{R}^N} \Big(|\nabla v_{n_k}|^{q(x)}+b(x) | v_{n_k}|^{q(x)}\Big)dx\Bigg)\\
		&\geq  C_{1} \| v_{n_k}\|_{b,p}^{p^{-}} +C_{2}\mathcal{H}(\kappa_\star^3)\| v_{n_k}\|_{b,q}^{q^{-}}.
	\end{align*}
	This implies that $\{v_{n_k}\}$ is bounded in $X_{G_1}$, This implies that $\| u_{nk}\|_{b,h} +\| v_{n_k}\|_{b,h}$ is bounded in $Z$.
	
In the sequel, we shall prove that $\{(u_{n_k} , v_{n_k})\}$ contains a subsequence converging strongly in $Z$.	
We note that the sequence $\{(u_{n_k}\}$  is bounded in $X_{G_1}$. Therefore, up to a subsequence, $u_{n_k} \rightharpoonup u$ in $X_{G_1}$ and $u_{n_k}\to u$ a.e. in $\mathbb{R}^N$.	
	\begin{align*}
		o(1)\| u_{n_k}-u\|_{b,h}&\geq \langle -E'_{\uplambda_{n_k}}(u_{n_k}-u,v_{n_k}),(u_{n_k}-u,0) \rangle	\\
		&=  K\left(
		\mathcal{B}(u_{n_k}-u)\right)\int_{\mathbb{R}^N } \Big( \mathcal{A}_1(\nabla (u_{n_k}-u)) .\nabla(u_{n_k}-u)  
		+  b(x)\mathcal{A}_2(u_{n_k}-u)(u_{n_k}-u)\Big)dx\\
		&\quad+  \int_{\mathbb{R}^N} |u_{n_k}-u|^{r(x)}dx+\int_{\mathbb{R}^N}\uplambda(x)\frac{ \partial \mathcal{F}}{\partial u}(x,u_{n_k}-u,v_{n_k})dx\\
		&\geq 	\mathfrak{K}_0\int_{\mathbb{R}^N}\Big(a_1(|\nabla u_{n_k}-u|^{p(x)})|\nabla u_{n_k}-u|^{p(x)} + b(x)a_2(|u_{n_k}-u|^{p(x)})| u_{n_k}-u|^{p(x)}\Big)dx \\
		&\quad+  \int_{\mathbb{R}^N} |u_{n_k}-u|^{r(x)}dx+\inf_{x\in \mathbb{R}^N}\uplambda(x)\int_{\mathbb{R}^N}\frac{ \partial \mathcal{F}}{\partial u}(x,u_{n_k}-u,v_{n_k})dx
		\\
		&\geq \mathfrak{K}_0  \int_{\mathbb{R}^N}\Big(a_1(|\nabla (u_{n_k}-u)|^{p(x)})|\nabla (u_{n_k}-u)|^{p(x)} + b(x)a_2(|u_{n_k}-u|^{p(x)})|u_{n_k}-u|^{p(x)}\Big)dx  \\
		&\geq  C_{1} \| u_{n_k}-u\|_{b,p}^{p^{-}} +C_{2}\mathcal{H}(\kappa_\star^3)\| u_{n_k}-u\|_{b,q}^{q^{-}}.    
	\end{align*}	
   This implies that $u_{n_k}$ converges strongly to $u$ in $X_{G_1}$.\par 
	
	 Next, we shall prove that there exists $v\in X_{G_1}$ such that $v_{n_k}\to v$ strongly in $X_{G_1}$.
	As $X_{G_1}$ is reflexive, passing to a subsequence, still denoted by $v_{n_k}$, we may assume that there exists $v\in X_{G_1}$ such that $v_{n_k}\rightharpoonup v$  in $X_{G_1}$ and $v_{n_k}(x)\to v(x)$ a.e. in  $\mathbb{R}^N$. we can also obtain that   $v_{n_k}\rightharpoonup v$ in $X_{G_1}$, as $k\to \infty$. 
	So there exists two positive and bounded measures $\upmu$ and  $\upnu$ on $\mathbb{R}^N$ and some at least countable family of points $(x_i)_{i\in I} \subset \mathcal{C}_{h}= \{x\in 	\mathbb{R}^N\colon r(x)=h^*(x)\}$ and of positive numbers $(\upnu_i)_{i\in I} $ and 
	$(\upmu_i)_{i\in I} $ such that 
	\begin{gather*}
		|\nabla u_{n_k}|^{h(x)}+ b(x) | u_{n_k}|^{h(x)}\overset{\ast}{\rightharpoonup}\upmu   \text{ in } \mathcal{M}_{B}(\mathbb{R}^N)
		\\
		\quad 
		| u_{n_k}|^{r(x)}\overset{\ast}{\rightharpoonup} \upnu  \text{
			in }\mathcal{M}_{B}(\mathbb{R}^N).
	\end{gather*}	
	
According to Theorem \ref{ccpp}, we have
	\begin{gather*} 
			\upmu = |\nabla u|^{h(x)} + b(x)| u|^{h(x)} + \sum_{i\in I} \upmu_i \delta_{x_i} +\tilde{\upmu}\quad \upmu(\mathcal{C}_{h})\leq1, \\
		\upnu=|u|^{r(x)} + \sum_{i\in I}\upnu_i\delta_{x_i}\quad \upnu(\mathcal{C}_{h})\leq C^{\ast}, 
	\end{gather*}
	where $\delta_{x_i}$ is the Dirac mass at  $x_i$, $I$ is countable index set  and $ \tilde{\upmu}$ is a non-atomic measure	
	\begin{gather} 
		\upnu(\mathcal{C}_{h})\leq 2^{\frac{h^+r^+}{h^-}}
		C^{\ast} \max \Big\{\upmu(\mathcal{C}_{h})^{\frac{r^+}{h^-}},\upmu(\mathcal{C}_{h})^{\frac{r^-}{h^+}}\Big\}, \label{g01}\\
		\upnu_i\leq 	C^{\ast} \max \{\upmu_i^{\frac{r^+}{h^-}},\upmu_i^{\frac{r^-}{h^+}}\},   \  \hbox{for all} \  i\in I. \label{g11}
	\end{gather}
	
	\noindent
	Concentration at infinity of the sequence $\{v_{n_k}\}$ is described by the following quantities:
	\begin{gather*} 
		\upmu_{\infty} := \lim_{R\to \infty} \limsup_{n_k\to \infty}
		\int_{\{x\in \mathbb{R}^N : |x|>R\}}( |\nabla v_{n_k}|^{h(x)}+ b(x) | v_{n_k}|^{h(x)})dx	, 		
		\\
		\upnu_{\infty} := \lim_{R\to \infty} \limsup_{n_k\to \infty}
		\int_{\{x\in \mathbb{R}^N :|x|>R\}} | v_{n_k}|^{r(x)}dx.
	\end{gather*}	
	
	We claim that $I$ is finite and for $i \in I$, either
	$v_i=0$ or \\ $\upnu_{i}\geq \max\Big\{\Big(\frac{(1-\mathcal{H}(\kappa_\star^3)) \min \left\{\kappa_1^0,\kappa_2^0\right\}+ \mathcal{H}(\kappa_\star^3)\min \left\{\kappa_1^2,\kappa_2^2\right\}}{ S^{\frac{h ^-}{r^+}}}\Big)^{\frac{r^+}{r^+-h^-}},\Big(\frac{(1-\mathcal{H}(\kappa_\star^3)) \min \left\{\kappa_1^0,\kappa_2^0\right\}+ \mathcal{H}(\kappa_\star^3)\min \left\{\kappa_1^2,\kappa_2^2\right\}}{ S^{\frac{h^-}{r^-}}}\Big)^{\frac{r^-}{r^--h^-}}\Big\}.$\\
	
	\noindent
	Let $x_i\in \mathcal{C}_{h}$ be a singular point of the measures $\upmu$ and $\upnu$. 
	We choose $\phi \in C^{\infty}_0( \mathbb{R}^N, \left[ 0,1\right]  )$ such that $\left| \nabla \phi \right|_{\infty}\leq 2$ and
	\begin{eqnarray*}
		\label{ineq1}
		\phi(x)=
		\begin{cases}
			1,  &\quad\text{if }  \left|x\right| <1,\\
			
			0, 	&\quad\text{if } \left|x\right|\geq 2. \\
		\end{cases}
	\end{eqnarray*}
	
	\noindent
	We define, for any $\varepsilon>0$  and $i\in I$, the function	 
	
	$$ 	 \phi_{i,\varepsilon} := \phi \Big(\frac{x-x_i}{\varepsilon}\Big),\quad \  \hbox{for all} \  x \in  \mathbb{R}^N.$$
	
	\noindent
	Note that $\phi_{i,\varepsilon} \in C^{\infty}_0( \mathbb{R}^N, \left[ 0,1\right]  )$, $|\nabla \phi_{i,\epsilon}|_\infty\leq \frac{2}{\varepsilon}$ and 
	\begin{eqnarray*}
		\label{ineq1}
		\phi_{i,\varepsilon}(x)=
		\begin{cases}
			1,  &  x\in  B_{\varepsilon}(x_i),\\
			
			0, 	&  x\in \mathbb{R}^N\setminus B_{2\varepsilon}(x_i). \\
		\end{cases}
	\end{eqnarray*}
   It is clear that $\{v_{n_k}\phi_{i,\varepsilon}\}$ is bounded in $X_{G_1}$. From this, we can conclude
    that $\langle E'_{\uplambda_{n_k}}(u_{n_k},v_{n_k}),(0,v_{n_k}	\phi_{i,\varepsilon})\rangle\rightarrow 0$ as $ n_k \rightarrow +\infty$, that is, 
	we obtain		
	\begin{align*}
		\langle E'_{\uplambda_{n_k}}(u_{n_k},v_{n_k}),(0,v_{n_k}	\phi_{i,\varepsilon})\rangle&= 	K\left(
		\mathcal{B}(v_{n_k})\right)\int_{\mathbb{R}^N } \Big( a_1(|\nabla v_{n_k}| ^{p(x)}) |\nabla v_{n_k}| ^{p(x)-2}\nabla v_{n_k}\nabla (v_{n_k}\phi_{i,\varepsilon})  \\
		&\quad + b(x)a_2(| v_{n_k}| ^{p(x)}) |v_{n_k}| ^{p(x)-2} v_{n_k} (v_{n_k}\phi_{i,\varepsilon})\Big) \,dx - \int_{\mathbb{R}^N  }|v_{n_k}|^{r(x)-2}v_{n_k} (v_{n_k}\phi_{i,\epsilon})\,dx 
		\\
		&\quad	- \int_{\mathbb{R}^N  }\uplambda (x)  \frac{ \partial \mathcal{F}}{\partial v}(x,u_{n_k},v_{n_k})v_{n_k}\phi_{i,\epsilon}\,dx   \rightarrow 0 \text{ as } n_k\rightarrow +\infty.
	\end{align*} 
	
	\noindent
	That is, 
	\begin{multline}\label{3.2}
		K\left(\mathcal{B}(v_{n_k})\right)\int_{\mathbb{R}^N  } a_1(|\nabla v_{n_k}| ^{p(x)}) |\nabla v_{n_k}| ^{p(x)-2}\nabla v_{n_k}\nabla \phi_{i,\varepsilon} v_{n_k} \,dx = -	K\left(
		\mathcal{B}(v_{n_k})\right)\int_{\mathbb{R}^N } \Big(a_1(|\nabla v_{n_k}| ^{p_(x)}) |\nabla v_{n_k}| ^{p(x)} \\
		+ b(x) a_2(| v_{n_k}| ^{p_(x)}) | v_{n_k}| ^{p(x)} \Big)\phi_{i,\varepsilon}\,dx +\int_{\mathbb{R}^N }|v_{n_k}|^{r(x)}\phi_{i,\epsilon}\,dx +\int_{\mathbb{R}^N  }\uplambda (x)  \frac{ \partial \mathcal{F}}{\partial v}(x,u_{n_k},v_{n_k})v_{n_k}\phi_{i,\epsilon}\,dx
		+ o_{n_k}(1).
	\end{multline} 
	
	Now, we shall prove that 	
	\begin{align}\label{3.3}
		\lim_{\varepsilon \to 0} \left\lbrace 
		\limsup_{n_k\to +\infty }	K\left(\mathcal{B}(v_{n_k})\right) \int_{\mathbb{R}^N  } a_1(|\nabla v_{n_k}| ^{p(x)}) |\nabla v_{n_k}| ^{p(x)-2}\nabla v_{n_k}\nabla \phi_{i,\varepsilon} v_{n_k} \,dx \right\rbrace  	& = 0.
	\end{align} 
	
	\noindent
	Note that, due to the hypotheses $(\textbf{\textit{H}}_{a_2})$ enough to show that	
	\begin{align}\label{3.4}
		\lim_{\varepsilon \to 0} \left\lbrace 
		\limsup_{n_k\to +\infty }	K\left(\mathcal{B}(v_{n_k})\right)\int_{\mathbb{R}^N  }  |\nabla v_{n_k}| ^{p(x)-2}\nabla v_{n_k}\nabla \phi_{i,\varepsilon} v_{n_k} \,dx \right\rbrace  	& = 0,
	\end{align} 
	and
	\begin{align}\label{3.5}
		\lim_{\varepsilon \to 0} \left\lbrace 
		\limsup_{n_k\to +\infty }	K\left(\mathcal{B}(v_{n_k})\right)\int_{\mathbb{R}^N  }  |\nabla v_{n_k}| ^{q(x)-2}\nabla v_{n_k}\nabla \phi_{i,\varepsilon} v_{n_k} \,dx \right\rbrace  	& = 0,
	\end{align}
	
	\noindent
	First, by using the H\"{o}lder inequality, we have
	\begin{align*}
		\left| \int_{\mathbb{R}^N}  |\nabla v_{n_k}| ^{p(x)-2}\nabla v_{n_k}\nabla\phi_{i,\varepsilon} v_{n_k}\,dx\right|
		&	\leq 2\left|  \left|\nabla v_{n_k} \right|^{p(x)-1}\right| _{\frac{p(x)}{p(x)-1}}\left| \nabla\phi_{i,\varepsilon} v_{n_k}\right| _{p(x)},
	\end{align*} 
  given that $\left\lbrace v_{n_k}\right\rbrace $ is bounded, the sequence of real values $\left| \left|\nabla v_{n_k} \right|^{p(x)-1}\right| _{\frac{p(x)}{p(x)-1}}$ is also bounded. Therefore, there exists a positive constant $C$ such that	
	\begin{align*}
		\left| \int_{\mathbb{R}^N  }  |\nabla  v_{n_k}| ^{p(x)-2}\nabla  v_{n_k}\nabla\phi_{i,\varepsilon}  v_{n_k} \,dx\right|
		&	\leq  C|\nabla\phi_{i,\varepsilon} v_{n_k}|_{p(x)}.
	\end{align*} 
	
	Moreover $\left\lbrace v_{n_k}\right\rbrace $ is bounded in $W^{1,p(x)}_{b}(B_{2\varepsilon}(x_i))$, then there exists a subsequence denoted again $\left\lbrace v_{n_k}\right\rbrace $  weakly converges to $v$ in $L^{p(x)}(B_{2\varepsilon}(x_i))$. Hence 	
	\begin{align*}
		\limsup_{n_k \to +\infty }	\left| \int_{\mathbb{R}^N }  |\nabla v_{n_k}| ^{p(x)-2}\nabla v_{n_k}\nabla\phi_{i,\varepsilon} v_{nk}\,dx\right|
		&	\leq  C|\nabla\phi_{i,\varepsilon} v_{n_k}|_{p(x)}\\
		&	\leq 2C \limsup_{\varepsilon \to 0}
		||\nabla\phi_{i,\varepsilon}|^{p(x)}|_{(\frac{p^{\star}(x)}{p(x)})^{'},B_{2\varepsilon}(x_i)} ||v|^{p(x)}|_{\frac{p^{\star}(x)}{p(x)},B_{2\varepsilon}(x_i)}	\\
		&	\leq 2C	\limsup_{\varepsilon \to 0 } ||\nabla\phi_{i,\varepsilon}|^{p(x)}|_{\frac{N}{p(x)},B_{2\varepsilon}(x_i)} ||v|^{p(x)}|_{\frac{N}{N-p(x)},B_{2\varepsilon}(x_i)}.
	\end{align*}	
	Note that	
	\[
	\int_{B_{2\varepsilon}(x_i) } 
	(|\nabla\phi_{i,\varepsilon}|^{p(x)})^{(\frac{p^{\star}(x)}{p(x)})'}dx = \int_{B_{2\varepsilon}(x_i) } |\nabla\phi_{i,\varepsilon}|^{N}dx
	\leq \Big(\frac{2}{\varepsilon}\Big)^N meas (B_{2\varepsilon}(x_i))=\frac{4^N}{N}\upomega_N,
	\]
	where $\upomega_N$ is the surface area of an $N$-dimensional unit sphere. Since  $ \int_{B_{2\varepsilon}(x_i) } (| v_{nk}|^{p(x)})^{\frac{p^{\star}(x)}{p(x)}}dx \to 0$ as $\varepsilon \to 0$, we can conclude that $|\nabla\phi_{i,\varepsilon}  v_{n_k}|_{p(x)}\to 0$, which implies
	\begin{equation}\label{con}
		\lim_{\varepsilon \to 0} \left\lbrace 
		\limsup_{n_k\to +\infty }\left| 
		\int_{\mathbb{R}^N  }  |\nabla  v_{nk}| ^{p(x)-2}\nabla  v_{nk}\nabla\phi_{i,\epsilon}  v_{nk} \,dx \right| \right\rbrace = 0.
	\end{equation}	
	
	Since  $\left\lbrace  v_{nk}\right\rbrace $  is bounded in $W_{b}^{1,p(x)}(\mathbb{R}^N )$, we may assume that  $\mathcal{B}(v_{nk}) \to t\geq 0 $ as $ n_k \to +\infty$. We note that  $K(t)$ is 
	continuous, we then have
	
	\[ 
	K\Big(\mathcal{B}(v_{nk}) \Big) \to K(t)\geq \mathfrak{K}_0>0, \quad \text{as } n_k \to +\infty. 
	\]     
	
	\noindent
	Hence, by relation \eqref{con}, we obtain
	\begin{equation}\label{3.7}
		\lim_{\varepsilon \to 0} \left\lbrace 
		\limsup_{n_k\to +\infty }K\Big(\mathcal{B}(v_{nk}) \Big)
		\int_{\mathbb{R}^N  }  |\nabla  v_{nk}| ^{p(x)-2}\nabla  v_{nk}\nabla\phi_{i,\epsilon}  v_{nk} \,dx\right\rbrace   = 0.
	\end{equation}
	\noindent
	Analogously, we verify relation \eqref{3.5}. Therefore, we conclude the proof
	of relation \eqref{3.3}.

	Similarly, we can also get
	\begin{equation}\label{3.8}
		\lim_{\varepsilon \to 0 }
		\int_{\mathbb{R}^N }\uplambda(x) \frac{ \partial \mathcal{F}}{\partial v}(x,u_{n_k},v_{n_k})\phi_{i,\epsilon} v_{n_k} dx=0, \text{ as } k \rightarrow +\infty.
	\end{equation}
		Indeed, by use H\"{o}lder's inequality with assumption $(\textit{\textbf{F}}_2)$ and since $0\leq \phi_{i,\varepsilon} \leq 1$,  we obtain	
	\begin{align*}
		\lim_{\varepsilon \to 0 }
		\int_{\mathbb{R}^N }\uplambda (x)\frac{ \partial \mathcal{F}}{\partial v}(x,u_{n_k},v_{n_k})\phi_{i,\epsilon} v_{n_k} dx &\leq 
		\lim_{\varepsilon \to 0 }\sup_{x\in \mathbb{R}^N}\uplambda(x)\int_{\mathbb{R}^N   }	\left(  f_{1}(x)| u_{n_k}|^{\ell(x)}+  f_{1}(x)| v_{n_k}|^{\ell(x)}\right) \phi_{i,\varepsilon} v_{n_k}dx,\\
		&\leq \lim_{\varepsilon \to 0 }\uplambda^+\int_{\mathbb{R}^N }	\left(   f_{1}(x)| u_{n_k}|^{\ell(x)}+  f_{1}(x)| v_{n_k}|^{\ell(x)}\right)|\phi_{i,\varepsilon} v_{n_k}|dx\\
		&\leq \lim_{\varepsilon \to 0 }c_1  \Big(
		|f_1|_{l(x)}||u_{n_k}|^{\ell}|_{h^{\star}(x)}+|f_2|_{\l(x)}||v_{n_k}|^{\ell}|_{h^{\star}(x)}\Big)|\phi_{i,\epsilon} v_{n_k}|_{{h^{\star}(x)}}.
	\end{align*}	
	The above propositions yield
	\begin{align*}
		\lim_{\varepsilon \to 0 }
		\int_{\mathbb{R}^N }\uplambda (x)\frac{ \partial \mathcal{F}}{\partial v}(x,u_{n_k},v_{n_k})\phi_{i,\epsilon} v_{n_k} dx &\leq \lim_{\varepsilon \to 0 }c_1
		\Big(
		|f_1|_{l(x)}\|u_{n_k}\|_{h(x)}^{\ell}+|f_2|_{l(x)}\|v_{n_k}\|_{h(x)}^{\ell}\Big)\| v_{k_n}\|_{h(x),B_{2\varepsilon}(x_i)} .	\end{align*}
	and this last goes to zero because of	
	$$|f_1|_{\ell(x)}\|u_{n_k}\|_{h(x)}^{\ell}+|f_2|_{\ell(x)}\|v_{n_k}\|_{h(x)}^{\ell}<\infty.$$

	Since $\phi_{i,\varepsilon}$ has compact support, going to the limit $n_k \to +\infty$ and letting $\varepsilon \to 0$ in relation \eqref{3.2}, from relations \eqref{3.3} and \eqref{3.4}, we get 
		\begin{align*}
		0 &=\upnu_i-\lim_{\varepsilon \to 0} \Biggl(
		\limsup_{n_k\to +\infty } K\Big(\mathcal{B}(v_{nk}) \Big)	\int_{\mathbb{R}^N } \Big(a_1(|\nabla v_{n_k}| ^{p_(x)}) |\nabla v_{n_k}| ^{p(x)}\phi_{i,\varepsilon} + b(x) a_2(| v_{n_k}| ^{p_(x)}) | v_{n_k}| ^{p(x)}\phi_{i,\varepsilon} \Big)\,dx\Biggr) , \\
		&\leq\upnu_i-\lim_{\varepsilon \to 0} \Bigg(
		\limsup_{n_k\to +\infty } \mathfrak{K}_0
		\Big( 	\int_{\mathbb{R}^N } \Big[ a_1(|\nabla v_{n_k}| ^{p_(x)}) |\nabla v_{n_k}| ^{p(x)}+b(x) a_2(| v_{n_k}| ^{p_(x)}) | v_{n_k}| ^{p(x)}\Big]\phi_{i,\varepsilon} \,dx\Big)	\Bigg),
	\end{align*}
	and by applying assumption $(\textbf{\textit{H}}_{a_2})$, we obtain
	\begin{multline}\label{4.11}	
		0\leq \upnu_i-\lim_{\varepsilon \to 0} \Bigg( 
		\limsup_{n_k\to +\infty } \mathfrak{K}_0
		\Big(\min \left\{\kappa_1^0,\kappa_2^0\right\}\int_{\mathbb{R}^N}\Big( |\nabla v_{n_k}|^{p(x)}+ b(x)| v_{n_k}|^{p(x)}\Big)dx\\
		+\min \left\{\kappa_1^2,\kappa_2^2\right\}\mathcal{H}(\kappa_\star^3) \Big(|\nabla u_{n_k}|^{q(x)}+ b(x)| u_{n_k}|^{q(x)}\Big)dx,\Big)\Bigg).
	\end{multline}	
	Note that, when $\kappa_\star^3=0$, we have $h(x)=p(x)$. Hence, from Theorem \ref{ccpp} and the aforementioned arguments, we obtain
	\begin{align*}
		0 &\leq \upnu_i-  \mathfrak{K}_0\min \left\{\kappa_1^0,\kappa_2^0\right\}\lim_{\varepsilon \to 0} 
		\int_{\mathbb{R}^N  }\phi_{i,\varepsilon}\,d\upmu \\
		&\leq \upnu_i -\mathfrak{K}_0\min \left\{\kappa_1^0,\kappa_2^0\right\}\Big(\upmu_{i} - \lim_{\varepsilon \to 0} 
		\int_{\mathbb{R}^N  }(|\nabla v|^{p(x)}+ b(x)| v|^{p(x)})\phi_{i,\varepsilon}\,dx\Big).
	\end{align*}	
	By using the Lebesgue dominated convergence theorem, we get
	$$\lim_{\varepsilon \to 0} 
	\int_{\mathbb{R}^N  }(|\nabla v|^{p(x)}+ b(x)| v|^{p(x)})\phi_{i,\varepsilon}\,dx=0.$$	
	Then, we get	
	\begin{equation}\label{k0}
	\mathfrak{K}_0\min \left\{\kappa_1^0,\kappa_2^0\right\}\upmu_{i}	\leq \upnu_i.
	\end{equation} 
	
	On the other hand, if $\kappa_\star^3>0$, we have  $h(x)=q(x)$ Therefore, it follows from Theorem \ref{ccpp} and relation \eqref{4.11} that
	\begin{align*}
		0 &\leq \upnu_i - \mathfrak{K}_0\min \left\{\kappa_1^0,\kappa_2^0\right\}\mathcal{H}(\kappa_\star^3)\lim_{\varepsilon \to 0} \left[ \limsup_{n_k \to \infty}
		\Big(\int_{\mathbb{R}^N }\Big(|\nabla v_{n_k}| ^{q(x)}+b(x)| v_{n_k}| ^{q(x)}\Big)\phi_{i,\varepsilon}\,dx\Big) \right]  \\
		&\leq \upnu_i - 	 \mathfrak{K}_0\min \left\{\kappa_1^2,\kappa_2^2\right\}\mathcal{H}(\kappa_\star^3)\lim_{\varepsilon \to 0} 
		\int_{\mathbb{R}^N  }\phi_{i,\varepsilon}\,d\upmu\\
		&\leq \upnu_i -	 \mathfrak{K}_0\min\left\{\kappa_1^2,\kappa_2^2\right\}\mathcal{H}(\kappa_\star^3)\Big(\upmu_{i} - \lim_{\varepsilon \to 0} 
		\int_{\mathbb{R}^N  }(|\nabla v|^{q(x)}+ b(x)| v|^{q(x)})\phi_{i,\varepsilon}\,dx\Big).
	\end{align*}	
and by applying the Lebesgue dominated convergence theorem, we obtain
	\begin{equation*}
		\lim_{\varepsilon \to 0} 
		\int_{\mathbb{R}^N  }(|\nabla v|^{q(x)}+ b(x)| v|^{q(x)})\phi_{i,\varepsilon}\,dx=0.
	\end{equation*}
	Then, we get
	\begin{equation}\label{k1}
		 \mathfrak{K}_0\min \left\{\kappa_1^2,\kappa_2^2\right\}\mathcal{H}(\kappa_\star^3)\upmu_{i}	\leq \upnu_i.
	\end{equation} 	
	Now, by combining relations \eqref{k0} and \eqref{k1}, we have
	\begin{equation}\label{mus}
	 \mathfrak{K}_0\Big((1-\mathcal{H}(\kappa_i^3)) \min \left\{\kappa_1^0,\kappa_2^0\right\}+ \mathcal{H}(\kappa_\star^3)\min \left\{\kappa_1^2,\kappa_2^2\right\}\Big)\upmu_{i} \leq \upnu_i.
	\end{equation}	
	Using relation \eqref{g11}, we obtain  
	$$	\upnu_i\leq 	C^{\ast} \max \Big\{\Big(\frac{\upnu_i}{ \mathfrak{K}_0 D}\Big)^{\frac{r^+}{h^-}},\Big(\frac{\upnu_i}{ \mathfrak{K}_0D}\Big)^{\frac{r^-}{h^+}}\Big\},$$
	where $D=(1-\mathcal{H}(\kappa_\star^3)) \min \left\{\kappa_1^0,\kappa_2^0\right\}+ \mathcal{H}(\kappa_\star^3)\min \left\{\kappa_1^2,\kappa_2^2\right\}$.	
	which implies that $\upnu_{i}=0$ or 
	\begin{equation}\label{chm1}\upnu_{i}\geq \max\Big\{\Big(\frac{ \mathfrak{K}_0 D}{ S^{\frac{h^-}{r^+}}}\Big)^{\frac{r^+}{r^+-h^-}},\Big(\frac{ \mathfrak{K}_0 D}{ S^{\frac{h^-}{r^-}}}\Big)^{\frac{r^-}{r^--h^-}}\Big\} 
	\end{equation}
	for all $i \in I$, which implies that $I$ is finite. The claim is therefore proved.	
	
	To analyze the concentration at $\infty$,  we choose a suitable cut-off function $\psi \in C_0^{\infty}(\mathbb{R}^N,[0,1])$ such that  $\psi(x)\equiv0$ on  $B_R(0)$ and $\psi(x)\equiv1$ on  $\mathbb{R}^N \backslash B_{2R}(0)$. We set $\psi_R(x)=\psi(\frac{x}{R})$, we can easily observe that  $\{v_{n_k}\psi_R\}$ is bounded in $X_{G_1}$ and 
	$\lim_{n_k\to \infty}\left\langle E'_{\uplambda_{n_k}}(u_{n_k},v_{n_k}),(0,v_{n_k}\psi_R)\right\rangle =0,$	
	\begin{align*}
		\langle E'_{\uplambda_{n_k}}(u_{n_k},v_{n_k}),(0,v_{n_k}	\psi_R)\rangle&= 	K\left(\mathcal{B}(v_{n_k})\right)\int_{\mathbb{R}^N }\Big( a_1(|\nabla v_{n_k}| ^{p(x)}) |\nabla v_{n_k}| ^{p(x)-2}\nabla v_{n_k}\nabla (v_{n_k}\psi_R)\\
		&\quad + b(x)a_2(| v_{n_k}| ^{p(x)}) |v_{n_k}| ^{p(x)-2} v_{n_k} (v_{n_k}\psi_R)\Big) \,dx - \int_{\mathbb{R}^N  }|v_{n_k}|^{r(x)-2}v_{n_k} (v_{n_k}\psi_R)\,dx 
		\\
		&\quad	- \int_{\mathbb{R}^N  }\uplambda(x)   \frac{ \partial \mathcal{F}}{\partial v}(x,u_{n_k},v_{n_k})v_{n_k}\psi_R\,dx   \rightarrow 0 \text{ as } n_k\rightarrow +\infty.
	\end{align*} 	
	In other words,
	\begin{multline}\label{3.2a}
		K\left(\mathcal{B}_1(v_{n_k})\right)\int_{\mathbb{R}^N  } a_1(|\nabla v_{n_k}| ^{p(x)}) |\nabla v_{n_k}| ^{p(x)-2}\nabla v_{n_k}\nabla \psi_R v_{n_k} \,dx = \int_{\mathbb{R}^N  }\uplambda(x) \frac{ \partial \mathcal{F}}{\partial v}(x,u_{n_k},v_{n_k})v_{n_k}\psi_R\,dx
		\\
	+ \int_{\mathbb{R}^N }|v_{n_k}|^{r(x)}\psi_R\,dx	-
		K\left(
		\mathcal{B}(v_{n_k})\right)\int_{\mathbb{R}^N } \Big(a_1(|\nabla v_{n_k}| ^{p_(x)}) |\nabla v_{n_k}| ^{p(x)}+b(x) a_2(| v_{n_k}| ^{p_(x)}) | v_{n_k}| ^{p(x)}\Big)\phi_{i,\varepsilon} \,dx +o_{n_k}(1).
	\end{multline} 	
As in the previous proof, we can find that 
	$\lim_{n_k\to \infty} \left| \nabla \psi_R v_{n_k}\right|_{p(x)}=0$ when $R\to \infty$, and 
	\begin{align*}
		\left| \int_{\mathbb{R}^N}  |\nabla v_{n_k}| ^{p(x)-2}\nabla v_{nk}\psi_R  v_{n_k}\,dx\right|
		&	\leq 2\left|  \left|\nabla v_{n_k} \right|^{p(x)-1}\right| _{\frac{p(x)}{p(x)-1}}\left| \nabla\psi_R  v_{n_k}\right|_{p(x)},
	\end{align*} 
	since $\left\lbrace v_{n_k}\right\rbrace $ is bounded, the real-valued sequence $ \left|  \left|\nabla v_{n_k} \right|^{p(x)-1}\right| _{\frac{p(x)}{p(x)-1}}$ is also bounded, then 
	
	\begin{equation}\label{kk0}
		\lim_{R\to \infty}\limsup_{n_k\to +\infty }K\left(
		\mathcal{B}(v_{n_k})\right)\int_{\mathbb{R}^N} | |\nabla v_{n_k}| ^{p(x)-2}\nabla v_{n_k}\psi_R  v_{n_k} |\,dx =0.
	\end{equation}	
	Similarly, we can also get 	
	\begin{equation}\label{kk1}
		\lim_{R\to \infty}\limsup_{n_k\to +\infty }K\left(
		\mathcal{B}(v_{n_k})\right)\int_{\mathbb{R}^N} | |\nabla v_{n_k}| ^{q(x)-2}\nabla v_{n_k}\psi_R  v_{n_k} |\,dx =0.
	\end{equation}	
	Therefore, we have 	
	$$\lim_{R\to \infty}\limsup_{n_k\to +\infty }\int_{\mathbb{R}^N  } a_1(|\nabla v_{n_k}| ^{p(x)}) |\nabla v_{n_k}| ^{p(x)-2}\nabla v_{n_k}\nabla \psi_R v_{n_k} \,dx =0.$$	
	Note that $v_{n_k}\to v$ weakly in $X_{G_1}$, so
	$\int_{\mathbb{R}^N}\uplambda(x)\frac{ \partial \mathcal{F}}{\partial v}(x,u,v)(v_{n_k}-v)\psi_Rdx\to 0$. As
	\begin{align*}
		\left| \int_{\mathbb{R}^N}\uplambda (x)(\frac{ \partial \mathcal{F}}{\partial v}(x,u_{n_k},v_{n_k})-\frac{ \partial \mathcal{F}}{\partial v}(x,u,v))v_{n_k}\psi_Rdx\right| &\leq c|(\frac{ \partial \mathcal{F}}{\partial v}(x,u_{n_k},v_{n_k})-\frac{ \partial \mathcal{F}}{\partial v}(x,u,v))\psi_R|_{(h^{\star}(x))'}|v_{n_k}|_{h^{\star}(x)},\\
		&\leq c|\frac{ \partial \mathcal{F}}{\partial v}(x,u_{n_k},v_{n_k})-\frac{ \partial \mathcal{F}}{\partial v}(x,u,v)|_{L^{(h^{\star}(x))'}(\mathbb{R}^N\backslash B_R(0))},
	\end{align*}
	According to assumption $(\textbf{\textit{F}}_1)$, analogous to Fu and Zhang to \cite[Theorem 4.3]{Fu0}, for any $\epsilon>0$, there exists $R_1>0$ such that when $R>R_1$, $|\frac{ \partial \mathcal{F}}{\partial v}(x,u_{n_k},v_{n_k})-\frac{ \partial \mathcal{F}}{\partial v}(x,u,v)|_{L^{(h^{\star}(x))'}(\mathbb{R}^N\backslash B_R(0))}<\epsilon$, for any $n\in \mathbb{N}$. 
	
	\noindent
	Note that $\int_{\mathbb{R}^N}\frac{ \partial \mathcal{F}}{\partial v}(x,u,v)v\psi_Rdx\to 0$ as $R\to \infty$. Thus we obtain that
	\begin{align*}
		\lim_{R\to \infty}&\limsup_{k\to +\infty }\int_{\mathbb{R}^N}\uplambda (x)\frac{ \partial \mathcal{F}}{\partial v}(x,u_{nk},v_{nk})v_{nk}\psi_Rdx \\
		\quad&\leq\sup_{ x\in\mathbb{R}^N}\uplambda (x) \lim_{R\to \infty}\limsup_{n_k\to +\infty }\int_{\mathbb{R}^N}\frac{ \partial \mathcal{F}}{\partial v}(x,u_{nk},v_{nk})v_{nk}\psi_Rdx
		\\
		\quad&=\uplambda^+\lim_{R\to \infty}\limsup_{n_k\to +\infty }\int_{\mathbb{R}^N}(\frac{ \partial \mathcal{F}}{\partial v}(x,u_{n_k},v_{n_k})-\frac{ \partial \mathcal{F}}{\partial v}(x,u,v))v_{n_k}\psi_R +\frac{ \partial \mathcal{F}}{\partial v}(x,u,v))(v_{n_k}-v)
		+\frac{ \partial \mathcal{F}}{\partial v}(x,u,v)v\psi_Rdx,\\
		\quad&=\uplambda^+\lim_{R\to \infty}\Big(\limsup_{n_k\to +\infty }\int_{\mathbb{R}^N}(\frac{ \partial \mathcal{F}}{\partial v}(x,u_{n_k},v_{n_k})-\frac{ \partial \mathcal{F}}{\partial v}(x,u,v))v_{n_k}\psi_Rdx +\int_{\mathbb{R}^N}\frac{ \partial \mathcal{F}}{\partial v}(x,u,v)v\psi_Rdx
		\Big),\\
		\quad&=\uplambda^+\Big(\lim_{R\to \infty}\limsup_{n_k\to +\infty }\int_{\mathbb{R}^N}(\frac{ \partial \mathcal{F}}{\partial v}(x,u_{n_k},v_{n_k})-\frac{ \partial \mathcal{F}}{\partial v}(x,u,v))v_{n_k}\psi_Rdx +\lim_{R\to \infty}\int_{\mathbb{R}^N}\frac{ \partial \mathcal{F}}{\partial v}(x,u,v)v\psi_Rdx\Big),\\
		\quad&=0.
	\end{align*}
	Since $\psi_{R}$ has compact support, going to the limit $n_k \to +\infty$ and letting $R \to \infty$ in relation \eqref{3.2a}, we get 
	\begin{equation*}\label{muss}
		\mathfrak{K}_0 \Big((1-\mathcal{H}(\kappa_\star^3)) \min \left\{\kappa_1^0,\kappa_2^0\right\}+ \mathcal{H}(\kappa_\star^3)\min \left\{\kappa_1^2,\kappa_2^2\right\}\Big)\upmu_\infty \leq \upnu_\infty.
	\end{equation*}	
	According to Theorem \ref{ccppp}, we have	
	either $\upnu_{\infty}=0$ or 
	\begin{equation}\label{chm2}
		\upnu_{\infty}\geq \max\left\{\Big(\frac{\mathfrak{K}_0 D}{ S^{\frac{h^-}{r^+}}}\Big)^{\frac{r^+}{r^+-h^-}},\Big(\frac{\mathfrak{K}_0 D}{ S^{\frac{h^-}{r^-}}}\Big)^{\frac{r^-}{r^--h^-}}\right\}
	\end{equation}
	for all $i \in I$, where $D= (1-\mathcal{H}(\kappa_\star^3)) \min \left\{\kappa_1^0,\kappa_2^0\right\}+ \mathcal{H}(\kappa_\star^3)\min \left\{\kappa_1^2,\kappa_2^2\right\}$.

	Next, we claim that relations \eqref{chm1} and \eqref{chm2} cannot occur. If the case \eqref{chm2}  holds, for some $i\in I$, then by using  $(\textbf{\textit{H}}_{a_4}), (\textbf{\textit{K}}_1)-(\textbf{\textit{K}}_2)$ and  $(\textbf{\textit{F}}_2)$, we get that 
	\begin{align*}
		c&=\lim_{n_k\to \infty}\Big(E_{\uplambda_{n_k}}(0,v_{n_k})-
		\langle E'_{\uplambda_{n_k}}(u_{n_k},v_{n_k}),(0,\frac{v_{n_k}}{
			\vartheta (x)} ) \rangle\Big)	\\
		&=\widehat{K}\left(
		\mathcal{B}(v_{n_k})\right) -\int_{\mathbb{R}^N}\frac{1}{r(x)}| v_{n_k}|^{r(x)}dx-\int_{\mathbb{R}^N}\uplambda(x)\mathcal{F}(x,0,v_{n_k})dx
	 -K\left(\mathcal{B}_1(v_{n_k})\right)\int_{\mathbb{R}^N }  \Big(\mathcal{A}_1(\nabla v_{n_k})\nabla(\frac{v_{n_k}}{\vartheta(x)})\\	&\quad + b(x)\mathcal{A}_2(v_{n_k}) \frac{v_{n_k}}{\vartheta(x)}\Big)dx  
		 + \int_{\mathbb{R}^N}\frac{1}{\vartheta(x)}| v_{n_k}|^{r(x)}dx+\int_{\mathbb{R}^N}\uplambda(x)\frac{ \partial \mathcal{F}}{\partial v}(x,0,v_{n_k})\frac{v_{n_k}}{\vartheta(x)}dx, \\
		&\geq \mathfrak{K}_0\int_{\mathbb{R}^N}\Big( \dfrac{\gamma }{\max\{\alpha_1,\alpha_2\}p(x)}-\frac{1}{\vartheta(x)}\Big)
		\Big[a_1(|\nabla v_{nk}|^{p(x)})|\nabla v_{nk}|^{p(x)}+b(x)a_2(| v_{nk}|^{p(x)})| v_{nk}|^{p(x)}\Big]dx\\
		&\quad +\mathfrak{K}_0\int_{\mathbb{R}^N}\dfrac{v_{nk}}{\vartheta(x)^2}a_1(|\nabla v_{nk}|^{p(x)})|\nabla v_{nk}|^{p(x)-2}\nabla v_{nk}\nabla\vartheta dx\\
		&\quad +\int_{\mathbb{R}^N}\Big( \dfrac{1}{\vartheta(x)}-\frac{1}{r(x)}\Big)| v_{nk}|^{r(x)}dx +\uplambda^-\int_{\mathbb{R}^N}\Big(\frac{ \partial \mathcal{F}}{\partial v}(x,0,v_{nk})\frac{v_{nk}}{\vartheta(x)}
		-\mathcal{F}(x,0,v_{nk})\Big)dx,\\
		&\geq \Big( \dfrac{1}{\vartheta^-}-\frac{1}{r^-}\Big)\upnu_{\infty}.     
	\end{align*}	
	So, by relation \eqref{chm2}, we have 	
	\begin{align*}
		c	&\geq \Big( \dfrac{1}{\vartheta^-}-\frac{1}{r^-}\Big)\max\Big\{\Big(\frac{ \mathfrak{K}_0 D}{ S^{\frac{h^-}{r^+}}}\Big)^{\frac{r^+}{r^+-h^-}},\Big(\frac{\mathfrak{K}_0 D}{ S^{\frac{h^-}{r^-}}}\Big)^{\frac{r^-}{r^--h^-}}\Big\}.
	\end{align*}
	This is impossible. Therefore, $\upnu_{\infty}=0$ for all $i\in I$. Similarly, we can prove that \eqref{chm1} cannot occur for
	any
	 $i$. Then
	$$ \limsup_{n_k \to +\infty }\int_{\mathbb{R}^N }| v_{nk}|^{r(x)}dx \to \int_{\mathbb{R}^N }| v|^{r(x)}dx.$$
	Note that if $|v_{n_k}-v|^{r(x)}\leq 2^{r^+}(|v_{n_k}|^{r(x)} + |v|^{r(x)}),$ then by the Fatou Lemma, we have	
	\begin{align*}
		\int_{\mathbb{R}^N }2^{r^+}|v|^{r(x)}dx&=	\int_{\mathbb{R}^N }\liminf_{n_k \to +\infty } \Big(2^{r^+}(|v_{n_k}|^{r(x)} + |v|^{r(x)}) -|v_{n_k}-v|^{r(x)}
		\Big)dx,\\
		&\leq \liminf_{n_k \to +\infty }	\int_{\mathbb{R}^N}\Big(2^{r^+}|v_{n_k}|^{r(x)} + 2^{r^+}|v|^{r(x)} -|v_{n_k}-v|^{r(x)}
		\Big)dx,\\
		&\leq	\int_{\mathbb{R}^N} 2^{r^{+}+1}|v|^{r(x)}dx - \limsup_{n_k \to +\infty }\int_{\mathbb{R}^N}|v_{n_k}-v|^{r(x)} dx.
	\end{align*}
	Thus,
	$ \int_{\mathbb{R}^N}|v_{n_k}-v|^{r(x)} dx \to 0,$  we have $v_{n_k} \to v$ strongly in $L^{r(x)}(\mathbb{R}^N).$
	
	Now, let us define the operator $\Phi$ as follows	
	$$\left[ \Phi(v),\tilde{v}\right]:= \int_{\mathbb{R}^N}\Big(\mathcal{A}_1(\nabla v)\nabla \tilde{v} + b(x)\mathcal{A}_2(v) \tilde{v}\Big)dx$$	
	for any $(v,\tilde{v})\in X_{G_1}\times X_{G_1} $. Using the Holder inequality and the condition $(\textit{\textbf{H}}_{a_2}),$ we can establish that 	
	$$\left| \left\langle \Phi(v),\tilde{v}\right\rangle \right| 
	\leq  c\left\| v\right\|^{q^--1}_{b,h}\left\| \tilde{v}\right\|_{b,h}.
	$$
	Thus, the linear functional
	$\Phi(v)$ is continuous on $X_{G_1}$ for each $v\in X_{G_1}$. Therefore, due to the weak convergence of  $v_{n_k}$ in $X_{G_1}$ , we obtain 
	\begin{equation}\label{keqa}
		\lim_{n_k\to\infty}\left\langle \Phi(v_{n_k}),v_0\right\rangle =\left\langle \Phi(v_{0}),v_0\right\rangle \quad \text{ and }
		\lim_{n_k\to\infty}\left\langle \Phi(v_{0}),v_{n_k}-v_0\right\rangle=0.
	\end{equation}	
	Clearly, $\left\langle \Phi(v_{n_k}),v_{n_k}-v_0\right\rangle\to 0$ as $n_k\to \infty$.  Hence, based on relation \eqref{keqa}, we can deduce that
	$$ \lim_{ n_k\to\infty}\left\langle \Phi(v_{n_k})-\Phi(v_{0}),v_{n_k}-v_0\right\rangle
	=\lim_{n_k\to\infty} \int_{\mathbb{R}^N}\Big(\mathcal{R}_n(x) +b(x)\mathcal{Q}_n(x)\Big)dx=0,
	$$	
	with
	$$\mathcal{R}_n(x)=\left\langle a_1(|\nabla v_{n_k}|^{p(x)})|\nabla v_{n_k}|^{p(x)-2}\nabla v_{n_k}- a_1(|\nabla v_0|^{p(x)})|\nabla v_0|^{p(x)-2}\nabla v, \nabla v_{n_k}-\nabla v_0\right\rangle$$	
	for all $x\in \mathbb{R}^N$ and  all $n \in \mathbb{N}$, and	
	$$\mathcal{Q}_n(x)=\left\langle a_2(|v_{n_k}|^{p(x)})|v_{n_k}|^{p(x)-2}v_{n_k}- a_2(|v_0|^{p(x)})|v_0|^{p(x)-2}v, v_{n_k}-v_0\right\rangle$$	
	for all $x\in \mathbb{R}^N$ and  all $n \in \mathbb{N}$. Hence, by applying some elementary inequalities (see, e.g., Hurtado et al. \cite[Auxiliary Results ]{Hurtado}), for any $ \upeta , \upxi \in \mathbb{R}^N ,$ 
	\begin{eqnarray}
		\label{ineq1}
		\begin{cases}
			|\upeta  -\upxi|^{p(x)}\leq c_{p}\left\langle a_i(|\upeta|^{p(x)})|\upeta|^{p(x)-2}\upeta  - a_i(|\upxi|^{p(x)})|\upxi|^{p(x)-2}\upxi ,\upeta -\upxi \right\rangle  &\quad\text{if } p(x) \geq2\\
			
			|\upeta -\upxi|^{2}\leq c (|\upeta| +|\upxi| )^{2-p(x)}\left\langle a_i(|\upeta|^{p(x)})|\upeta|^{p(x)-2}\upeta  - a_i(|\upxi|^{p(x)})|\upxi|^{p(x)-2}\upxi 
			,\upeta -\upxi\right\rangle	&\quad\text{if } 1<p(x) <2 \\
		\end{cases}
	\end{eqnarray}	
 By replacing $\upeta$ and $\upxi$ with $\nabla v_{n_k}$ and $\nabla v_0$ respectively and integrating over $\mathbb{R}^N$, we obtain 	
	$$\int_{\mathbb{R}^N} \mathcal{R}_n(x)dx \geq 
	C\int_{\left\lbrace x\in \mathbb{R}^N ; ~p(x)\geq 2\right\rbrace }|\nabla v_{n_k} -\nabla v_0|^{p(x)}dx
	$$		
	Thus, 
	\begin{equation} \label{PN1}
		\lim_{n_k\to \infty}\int_{\left\lbrace x\in \mathbb{R}^N ; ~p(x)\geq 2\right\rbrace }|\nabla v_{n_k} -\nabla v_0|^{p(x)}dx=0.
	\end{equation}

	On the other hand, by using relation \eqref{ineq1}, we get	
	$$\int_{\mathbb{R}^N} \mathcal{R}_n(x)dx \geq 
	C\int_{\left\lbrace x\in \mathbb{R}^N ;~ 1<p(x)< 2\right\rbrace }\sigma_1(x)^{p(x)-2}|\nabla v_{n_k} -\nabla v_0|^{2}dx
	$$		
	where $\sigma_1(x) =C(|v_{n_k}|+|\nabla v_0|)$. Therefore by  H\"{o}lder's inequality, we have
	\begin{multline*}
		\int_{\left\lbrace x\in \mathbb{R}^N ;~ 1<p(x)< 2\right\rbrace }|\nabla v_{n_k} -\nabla v_0|^{p(x)}dx
		= 	\int_{\left\lbrace x\in \mathbb{R}^N ;~ 1<p(x)< 2\right\rbrace }\sigma_1^{\frac{p(x)(p(x)-2)}{2}}\Big(\sigma_1^{\frac{p(x)(p(x)-2)}{2}}|\nabla v_{n_k} -\nabla v_0|^{p(x)} \Big)dx	\\
		\leq C \Vert \sigma_1^{\frac{p(x)(2-p(x))}{2}} \Vert_{L^{\frac{2}{2-p(x)}}(\left\lbrace x\in \mathbb{R}^N ;~ 1<p(x)< 2\right\rbrace)}
		\Vert \sigma_1^{\frac{p(x)(p(x)-2)}{2}}|\nabla v_{n_k} -\nabla v|^{p(x)} \Vert_{L^{\frac{2}{p(x)}}(\left\lbrace x\in \mathbb{R}^N ; ~1<p(x)< 2\right\rbrace)}\\
		\leq C\max\Biggl\{ \Vert \sigma_1\Vert^{\big(\frac{p(x)(p(x)-2)}{2}\big)^-}_{L^{p(x)}(\left\lbrace x\in \mathbb{R}^N ; ~1<p(x)< 2\right\rbrace)}, \Vert \sigma_1\Vert^{\big(\frac{p(x)(p(x)-2)}{2}\big)^+}_{L^{p(x)}(\left\lbrace x\in \mathbb{R}^N ; ~1<p(x)< 2\right\rbrace)}
		\Biggr\} \times\\
		\max\Biggl\{ \Big( 	\int_{\left\lbrace x\in \mathbb{R}^N ; ~1<p(x)< 2\right\rbrace}\sigma_1^{p(x)-2}|\nabla v_{n_k} -\nabla v_0|^{2}dx\Big)^{\frac{p^-}{2}}, \Big( 	\int_{\left\lbrace x\in \mathbb{R}^N ; ~ 1<p(x)< 2\right\rbrace}\sigma_1^{p(x)-2}|\nabla v_{n_k} -\nabla v_0|^{2}dx\Big)^{\frac{p^+}{2}}
		\Biggr\},
	\end{multline*}	
  As the last term on the right-hand side of the above inequality tends to zero, we can conclude	
	\begin{equation} \label{PN2}
		\lim_{n_k\to \infty}\int_{\left\lbrace x\in \mathbb{R}^N ; ~1<p(x)< 2\right\rbrace }|\nabla v_{n_k} -\nabla v_0|^{p(x)}dx=0.
	\end{equation}	
  Now, combining relation \eqref{PN1} with relation \eqref{PN2}, we obtain	
	\begin{equation*} 
		\lim_{n_k\to \infty}\int_{ \mathbb{R}^N }|\nabla v_{n_k} -\nabla v_0|^{p(x)}dx=0.
	\end{equation*}	
	The same arguments can be used to prove that
	\begin{equation*} 
		\lim_{n_k\to \infty}\int_{ \mathbb{R}^N }b(x)|v_{n_k} - v_0|^{p(x)}dx=0.
	\end{equation*}
   In conclusion, we have shown that the sequence $\left\lbrace v_{n_k}\right\rbrace$ converges strongly to $v_0$ in $X_{G_1}$. Therefore, we can conclude that $\left\lbrace \left( u_{n_k},v_{n_k}\right) \right\rbrace $ contains a subsequence converging strongly in $Z$.	
\end{proof}

Now, we are ready to prove Theorem \ref{theooo}.

\begin{proof} 
	The proof immediately follows from Theorem \ref{thhh}. 
	More precisely, it suffices to check the conditions of Theorem \ref{thhh}. Set
		$$Z=U\oplus V, \quad U=X_{G_1}\times \left\lbrace 0\right\rbrace , \quad  V= \left\lbrace 0\right\rbrace\times X_{G_1},$$
	and 
	$$ Y_0 =\left\lbrace 0\right\rbrace\times X_{G_1}^{(m)^\perp}, \quad Y_1=\left\lbrace 0\right\rbrace\times X_{G_1}^{(k)},$$
	where $m$ and $k$ are yet to be determined.	
	
\noindent	
Define a group action $G=\left\lbrace 1,\tau\right\rbrace \cong \mathbb{Z}_2$ by setting $\tau(u,v)=(-u,-v)$, then $\text{Fix}G=\left\lbrace 0\right\rbrace \times\left\lbrace 0\right\rbrace$  (also denote$\left\lbrace  0\right\rbrace$). 
It is clear that $U$ and $V$ are $G$-invariant closed subspaces of $Z$, and $Y_0$ and $Y_1$ are $G$-invariant closed subspaces of $V$  and 
$\text{codim}_{V}Y_0=m$, $\text{dim}Y_1=k$.

Let 
$$ \Sigma:= \left\lbrace  A\subset Z\backslash\left\lbrace  0\right\rbrace : \text{ A is closed in } X \text{ and } (u,v)\in A \implies (-u,-v)\in A\right\rbrace .$$

\noindent
Define an index $\chi$ on $\Sigma$ by 
$$\chi(A)=\begin{cases}
	~	\min \{ N\in \mathbb{Z} : \exists h\in  C(A,\mathbb{R}^N\backslash\left\lbrace  0\right\rbrace)  \text{ such that } h(-u,-v)=h(u,v) \},\\
	~ 0 \text{ if } ~~ A=\emptyset,\\
	~+\infty \text{ if such } h \text{ does not exist.}
\end{cases}$$
Then from Huang and Li \cite{Huang}, we deduce that $\chi$  is an index theory satisfying the properties given in Definition \ref{diif}. Moreover, $\chi$ satisfies the one-dimensional property. According to Definition \ref{diiif} , we can obtain a limit index $\chi^\infty$
with respect to $(Z_n)$ from $\chi$.
	
Now we shall verify the conditions of Theorem \ref{thhh}.	
	It is
	 easy to verify that
	  the conditions ($\textbf{\textit{B}}_1$), ($\textbf{\textit{B}}_2$), ($\textbf{\textit{B}}_4$) in Theorem \ref{thhh} are satisfied. Set 
	$$V_j=X_{G_1}^{(j)}=\text{span}\{e_1,e_2,...,e_j\}.$$ 
	Hence $(\textbf{\textit{B}}_3)$ in Theorem \ref{thhh} is also satisfied. In the
	sequel, we 
	shall verify the condition
	 $(\textbf{\textit{B}}_7)$ in Theorem \ref{thhh}. Note that $\text{Fix}G=\left\lbrace 0\right\rbrace $, which implies that $\text{Fix}G\cap V=\left\lbrace (0,0)\right\rbrace $, satisfying condition $(1)$ of $(\textbf{\textit{B}}_7)$. Now, we need to verify the conditions $(2)$ and $(3)$ of ($\textbf{\textit{B}}_7$).
	
	\noindent
	Hereafter, we shall focus our attention on the case when $z=(u,v)\in Z$ satisfies  $\left\| u\right\| _{b,h}\leq1$  and  $\left\| v\right\| _{b,h}\leq1$.\\
	
	$(i)$   Let $(0,v)\in Y_0\cap S_{\rho_m}(0)$ (where $\rho_m$ is yet to be determined). Thus by using assumptions ($\textit{\textbf{F}}_1$) and ($\textit{\textbf{H}}_{a_3}$), we have  	
	\begin{align*}
		E_{\uplambda}(0,v)&=  \widehat{K}\left(
		\mathcal{B}(v)\right)-\int_{\mathbb{R}^N}\frac{1}{r(x)}| v|^{r(x)}dx-\int_{\mathbb{R}^N}\uplambda (x)\mathcal{F}(x,0,v)dx,\\
		&\geq\gamma \mathfrak{K}_0 \int_{\mathbb{R}^N}\Big(\frac{A_1(| v |^{p(x)})}{p(x)} + b(x)\frac{A_2(| v |^{p(x)})}{p(x)}\Big)dx-\dfrac{1}{r^-}\int_{\mathbb{R}^N  }|v|^{r(x)}dx- \sup_{x\in \mathbb{R}^N}\uplambda(x) \int_{\mathbb{R}^N  }\dfrac{f_2(x)}{\ell(x)}|v|^{\ell(x)}dx,\\
		&\geq C\left[\left\| v \right\|_{h ,p}^{p^+}+\mathcal{H}(\kappa_\star)\left\| v \right\|_{h,q}^{q^+}\right] -\dfrac{1}{r^-}\int_{\mathbb{R}^N  }|v|^{r(x)}dx- \uplambda^+ \int_{\mathbb{R}^N  }\dfrac{f_2(x)}{\ell(x)}|v|^{\ell(x)}dx.
	\end{align*}
	Denote
	$$ \delta_{m}=\sup_{v\in X_{G_1}^{m^\perp},\left\| v\right\| _{b,h}\leq1} \int_{\mathbb{R}^N}\dfrac{f_2(x)}{\ell(x)}|v|^{\ell(x)}dx \quad\text{ and }   \tau_m=\sup_{v\in X_{G_1}^{m^\perp},\left\| v\right\| _{b,h}\leq1} \int_{\mathbb{R}^N}|v|^{r(x)}dx,
	$$	
	We invoke here Fan and Han \cite[ Lemma 3.3]{Fan0} to obtain that $
	\delta_m \to 0$, as $m\to \infty$.\par 
	
	Next, we need to verify that $\tau_m\to 0$ as $m\to \infty$. We know that $0 \leq \tau_m + 1 \leq \tau_m$, which implies that $\tau_m\to \tau\geq 0$ as $m\to \infty$. Therefore, there exist $v_m \in X_{G_1}^{m^\perp}$ such that 
	$$ 0\leq \tau_m -\int_{ \mathbb{R}^N }|v_m|^{r(x)}dx < \frac{1}{m},$$	
	for every $m=1,2,\dotsi$. As $X_{G_1}$ is reflexive, we can pass to a subsequence, still denoted by $\left\lbrace v_m\right\rbrace$, such that there exists $v\in X_{G_1}$ satisfying $v_m \rightharpoonup v $ weakly in $X_{G_1}$ as $m\to \infty$.\par 
	
	We claim $v=0$. in fact, for any 
	$ e_k^*\in \left\lbrace e_1^*, e_2^*, \dots, e_m^*,\dots\right\rbrace $, we have $ e_k^*(v_m)=0$ when $ m>k$, which implies that $ e_k^*(v_m)\to0$ as $m\to \infty$. It is immediate that $e_k^*(v)=0$ for any $k\in \mathbb{N}$. Since $(X_{G_1})^*= \overline{\text{span}\left\lbrace e_1^*,e_2^*,\dots,e_k^*,\dots \right\rbrace },$ we can conclude that $v=0$.
	
	\noindent
	By Theorem \ref{ccpp}, there exist a finite measure $\upnu$ and sequences  $\left\lbrace x_i\right\rbrace \subset \mathcal{C}_h$  such that  $|v_m|^{r(x)}\overset{\ast}{\rightharpoonup}\upnu = \sum_{i\in I}\upnu_i\delta_{x_i}$ 
	in $\mathcal{M}_{B}(\mathbb{R}^N)$. where $I$ is a countable set.
	Following a similar discussion as in Lemma \ref{lemma2}, we can conclude that $v_i=v(\left\lbrace x_i\right\rbrace )=0$ for any $i\in I$ where $x_i\neq0$.
	\noindent
	On the other hand, for any $0<t<R$, take 
	$\theta \in C_0^{\infty}(B_{2R}(0))$ such that 
	$0\leq \theta \leq 1$;  $\theta \equiv 1$ in 
	$B_{2R}(0)\backslash B_{2t}(0)$, $\theta \equiv0$ in $B_{t}(0)$. Then 	
	$$ \int_{ \mathbb{R}^N } \left| v_m\right| ^{r(x)}\theta dx \longrightarrow  \int_{ \mathbb{R}^N } \theta d\upnu = \int_{\left\lbrace  x\in \mathbb{R}^N;~~ t \leq \left| x\right|  \leq R\right\rbrace }\theta d\upnu=0,  \text{ as } m\to \infty.$$	
	Since
	$$ \int_{\left\lbrace  x\in \mathbb{R}^N; ~~ 2t\leq \left| x\right|  \leq 2R\right\rbrace } \left| v_m\right|^{r(x)} dx \leq \int_{ \mathbb{R}^N } \left| v_m\right| ^{r(x)}\theta dx,
	$$
	we obtain $\lim_{m\to \infty }\int_{ \mathbb{R}^N } \left| v_m\right| ^{r(x)}dx=0$.
	Therefore $\tau_m\to 0$, as $m\to \infty$. 
	
	\noindent
	Then we have 
	\begin{align*}
		E_{\uplambda}(0,v)
		&\geq C\left[\left\| v \right\|_{b,p}^{p^+}+\mathcal{H}(\kappa_\star^3)\left\| v \right\|_{b,q}^{q^+}\right] -\frac{\tau_m^{r^-}}{r^-}\left\| v \right\|_{b,h}^{r^-} -\uplambda^+ \delta_m^{\ell^-}\left\| v \right\|_{b,h}^{\ell^-}\\
		&\geq c\left\| v \right\|_{b,h}^{q^+} -\tau_m^{r^-}\left\| v \right\|_{b,h}^{r^-} -\uplambda^+ \delta_{m}^{\ell^-}\left\| v \right\|_{b,h}^{\ell^-},\\
		&\geq c\left\| v \right\|_{b,h}^{q^+} -(\tau_m^{r^-}+\uplambda^+ \delta_{m}^{\ell^-})\left\| v \right\|_{b,h}^{r^-}.
	\end{align*}

Let $$\rho_m=\Big(\frac{cq^+}{r^-(\tau_m^{r^-}+\uplambda \delta_m^{\ell^-})}\Big)^{\frac{1}{r^--q^+}}.$$ 
When $(0,v)\in Y_0\cap S_{\rho_m}(0)$ and $\left\| v \right\|_{b,h}=\rho_m$,  for sufficiently large $m$, we have
\begin{align*}
	\sup E_{\uplambda}(0,v)|_{Y_0\cap S_{\rho_m}(0)}&\geq
	\Big(\frac{cq^+}{r^-}\Big)^{\frac{r^-}{r^--q^+}}\Big(\frac{r^--q^+}{q^+}\Big) \Big(\frac{1}{\tau_m^{r^-}+\uplambda^+ \delta_{m}^{\ell^-}}\Big)^{\frac{q^+}{r^--q^+}},
\end{align*}
	where $\tau_m$ and $\delta_{m}\to 0$ as $m\to \infty$, thus we have 	
	\begin{align*}
		\sup E_{\uplambda}(0,v)|_{Y_0\cap S_{\rho_m}(0)}
		&\geq \Big(\frac{cq^+}{r^-}\Big)^{\frac{r^-}{r^--q^+}}\Big(\frac{r^--q^+}{q^+}\Big) \Big(\frac{1}{\tau_m^{r^-}+\uplambda^+ \delta_{m}^{\ell^-}}\Big)^{\frac{q^+}{r^--q^+}} =\mathfrak{M}_m\to \infty \text{ as } m\to \infty,
	\end{align*}
	that is, the condition $(2)$ of ($\textit{\textbf{B}}_7$) holds.\\
	
	$(ii)$ By $(\textit{\textbf{K}}_1)$ and $(\textit{\textbf{K}}_2)$, for any  $u\in X_{G_1}$ we have 
	\begin{align*}
		E_{\uplambda}(u,0)&=  - \widehat{K}\left(
		\mathcal{B}(u)\right) -\int_{\mathbb{R}^N}\frac{1}{r(x)}| u|^{r(x)}dx-\int_{\mathbb{R}^N}\uplambda(x)\mathcal{F}(x,0,v)dx,\\ 
		&\leq 0.
	\end{align*} 
	Hence,  we can choose $\mathfrak{M}$ such that	
	\begin{equation}\label{MMAaa}\mathfrak{M} > \sup_{u\in X_{G_1}}E_{\uplambda}(u,0).
	\end{equation}	
	On the other hand, from $(\textit{\textbf{K}}_2)$, we can obtain for $\xi>\xi_0$
	\begin{equation}\label{MMA}
		\widehat{K}(\xi)\leq \frac{\widehat{K}(\xi_0)}{\xi_0^\frac{1}{\gamma}}\xi^\frac{1}{\gamma}\leq c\xi^\frac{1}{\gamma}.
	\end{equation}	
	About the latter condition and relation \eqref{MMAaa}, for all $(u,v)\in U\oplus Y_1$, we have 	
	\begin{align*}
		E_{\uplambda}(u,v)&= 
		- \widehat{K}\left(
		\mathcal{B}(u)\right) + \widehat{K}\left(
		\mathcal{B}(v)\right) -\displaystyle\int_{\mathbb{R}^N  }\dfrac{1}{r(x)}|u|^{r(x)}dx -\displaystyle\int_{\mathbb{R}^N  }\dfrac{1}{r(x)}|v|^{r(x)}dx -\displaystyle\int_{\mathbb{R}^N}\uplambda(x)\mathcal{F}(x,u,v)dx,\\
		&\leq c\left\| v\right\|_{b,h}^{\frac{q^+}{\gamma}}-c\left|  v \right|_{r(x)}^{r^-}+\mathfrak{M} .
	\end{align*} 	
	Since $\left|  . \right|_{r(x)}$ is also a norm on $Y_1$, and 
	 $Y_1$is a finite-dimensional space, thus $\left\| .\right\|_{b,h}$ and  $\left|  . \right|_{r(x)}$ are equivalent. Then, we get
	\begin{align*}
		E_{\uplambda}(u,v)&\leq c\left\| v \right\|_{b,h}^{\frac{q^+}{\gamma}}-c_{p^*}\left\| v \right\|_{b,h}^{r^-}  +\mathfrak{M}. 
	\end{align*} 
 Given that $\gamma>\frac{q^+}{r^-}$, we have	
	$$\sup E_{\uplambda}|_{U\oplus Y_1} <+\infty.$$
	Therefore, we can choose $k>m$ and $\mathfrak{N}_k >\mathfrak{M}_m$ such that	
	$$E_{\uplambda}| _{U\oplus Y_1} \leq\mathfrak{N}_k,$$
	which satisfies the condition $(3)$ in ($\textit{\textbf{B}}_7$). According to Lemma \ref{lemma2}, $E_{\uplambda}(u,v)$ satisfies the condition of $(PS)_c$ for any $c\in \left[ \mathfrak{M}_m, \mathfrak{N}_k \right] $, thus $(\textbf{\textit{B}}_6)$ in Theorem \ref{thhh} holds. Consequently, based on Theorem \ref{thhh}, we can conclude that
	\[
	c_i=\sup_{\chi^{\infty}(A)\leq i} \sup_{ z=(u,v)\in A}E_{\uplambda}(u,v), \quad  -k+1\leq i\leq -m,
	\]
   represent critical values of $E_{\uplambda}$, where $\mathfrak{M}_m \leq c_{-k+1}\leq ... \leq c_{-m}\leq \mathfrak{N}_k$. As we let $m\to \infty$, we can obtain an unbounded sequence of critical values $c_i$. Due to the even nature of the functional $E$, this results in two critical points $\mp z_i$ of $E_{\uplambda}$ corresponding to each $c_i$.
   \end{proof}

\section*{Acknowledgements}
Repov\v{s} was supported by the Slovenian Research and Innovation Agency grants P1-0292, J1-4031, J1-4001, N1-0278, N1-0114, and N1-0083. The authors acknowledge the referees for their comments and suggestions.

\end{document}